\documentclass{mcom-l}
\usepackage[usenames]{color}
\usepackage[utf8x]{inputenc}
\usepackage{hyperref}
\newtheorem{theorem}{Theorem}

\newtheorem{preliminary}[theorem]{Preliminary}

\theoremstyle{definition}
\newtheorem{definition}{Definition}

\newcommand{\N}{\mathbb{N}}
\newcommand{\NZ}{\N_0}
\glossary{$\NZ$: natural numbers including 0}

\newcommand{\R}{\mathbb{R}}
\newcommand{\C}{\mathbb{C}}

\newcommand{\wo}{\setminus}

\newcommand{\abs}[1]{\left|#1\right|}

\usepackage{graphics}
\usepackage{rotating}
\newcommand \sx {\scalebox}
\newcommand \rme {\mathrm{e}}
\newcommand \rot {\begin{rotate}}
\newcommand \ero {\end{rotate}}
\newcommand \iL[1] {\label{#1}}
\newcommand \ing {\includegraphics}
\newcommand \rmi {\mathrm{i}}
\newcommand \be {\begin{eqnarray}}
\newcommand \ee {\end{eqnarray}}
\newcommand \eL[1] {\iL{#1} \end{eqnarray}}

\newcommand \rf [1] {(\ref{#1})}

\usepackage{graphicx}

\newcommand \JP [1] {}
\begin{document}
\title[base exp(1/e)]{Computation of the Two Regular Super-Exponentials to base exp(1/e)}

\author{Henryk Trappmann~}
\address{ }
\email{henryk@pool.math.tu-berlin.de}

\author{~Dmitrii Kouznetsov}
\address{Institute for Laser Science, University of Electro-Communications
1-5-1 Chofugaoka, Chofushi, Tokyo, 182-8585, Japan}
\email{dima@uls.uec.ac.jp}
\thanks{}

\subjclass[2010]{Primary}{30D05}
\subjclass[2010]{Secondary}{30A99}
\subjclass[2010]{Secondary}{33F99}
\subjclass[2010]{Secondary}{65Q20}
\date{\today}
\dedicatory{}
\begin{abstract}
The two regular super-exponentials to base exp(1/e) are constructed.
An efficient algorithm for the evaluation of these super-exponentials
and their inverse functions is suggested and compared to the already 
published results. 
\end{abstract}

\maketitle
\section{Introduction}
We call a holomorphic function $F$ a superfunction \cite{qfac} of some base function $f$ if
it is a solution of the equation
\be
F(z\!+\!1)=f(F(z))
\eL{f}
In the case $f\!=\!\exp_{b}$, i.e.\ for the exponential base function
$f(z)\!=\!b^{z}$, we call $F$ {\em super-exponential} to base $b$. 
In addition, if the super-exponential $F$ satisfies the equation 
\be
F(0)=1
\eL{F01}
we call it {\em tetrational}; for integer values of the argument $z$,
equation \eqref{f} and \eqref{F01} implies $F$ to be the $z$ times
application of the exponential $\exp_b$ to unity
\be
F(z)=\underbrace{\exp_{b}\!\!\Big(\exp_{b}\!\big(
  ... \exp_{b}(}_{z\times}1) ... \big)\Big).
\eL{integerz}

Conversely a function $A$ is called Abel function of some base
function $f$ if it satisfies
\begin{align}
A(f(z))=A(z)\!+\!1\iL{abel0}.
\end{align}
For $f(z)=b^z$ we call $A$ {\em super-logarithm} to base $b$. The inverse of a
super-exponential is a super-logarithm.
(In some ranges of values of $z$, the relations $F(A(z))=z$ and $A(F(z))=z$ hold.)




We have constructed super-exponentials and efficient algorithms of
their numerical evaluation in \cite{citeulike:4195962} for $1\!<\!b\!<\!\rme^{1/\rme}$,
and in \cite{kouznetsov:sqrt2} for $b\!>\!\rme^{1/\rme}$. However,
neither method used in these publications is applicable to base
$b\!=\!\rme^{1/\rme}$. Especially this case is analyzed by Walker in
\cite{walker123}; he evaluates the two Abel functions at several points in the
complex plane. 
Here we show that his constructions are equal to the two regular Abel
functions ({\em regular} in the sense of Szekeres \cite{szekeres:regular}) and suggest a
faster/more precise alternative algorithm (which goes back to
Écalle) that
allows to plot the complex maps in real time.

\JP{
The straightforward iteration of the equation with the Cauchy integral becomes unstable while the asymptotic values at $\pm i \infty$ are the same, and tend to return the constant value $\rme$ which is fixed point of the base function. The representation through the Schroeder function (..) becomes invalid because of singularity of the Schroeder coefficients. 
}


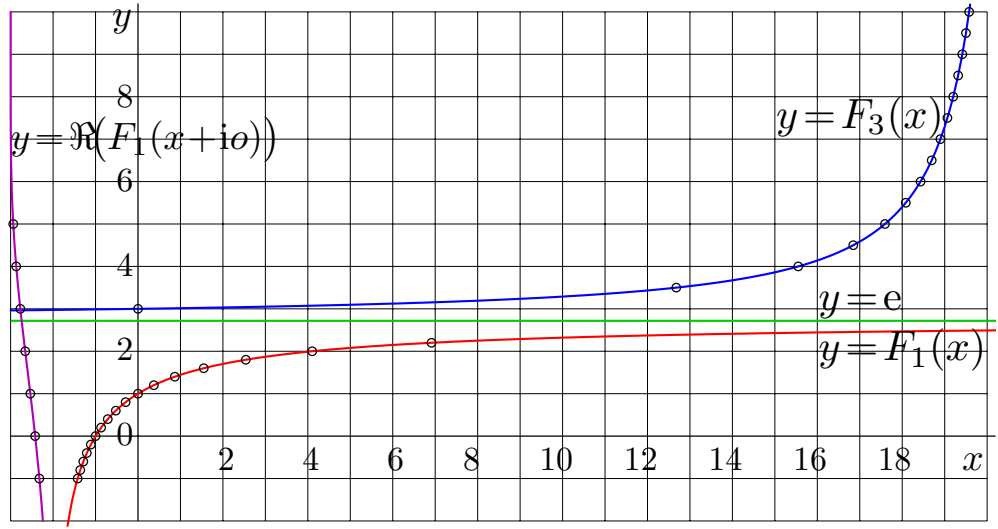
\begin{figure}\begin{center}{
\sx{1.6}{\begin{picture}(340,126)
\put(0,0){\ing{fige1e1}}
\put( 26,119){\sx{.9}{$y$}}
\put( 27,100){\sx{.8}{$8$}}
\put( 27,80){\sx{.8}{$6$}}
\put( 27,60){\sx{.8}{$4$}}
\put( 27,40){\sx{.8}{$2$}}
\put( 27,20){\sx{.8}{$0$}}
\put( 51,14){\sx{.8}{$2$}}
\put( 71,14){\sx{.8}{$4$}}
\put( 91,14){\sx{.8}{$6$}}
\put(109,14){\sx{.8}{$8$}}
\put(127,14){\sx{.8}{$10$}}
\put(147,14){\sx{.8}{$12$}}
\put(167,14){\sx{.8}{$14$}}
\put(187,14){\sx{.8}{$16$}}
\put(207,14){\sx{.8}{$18$}}
\put(227,14){\sx{.9}{$x$}}
\put(2,90){\sx{.9}{$y\!=\!\Re\!\big(F_{1}(x\!+\!\rmi o)\big)$}}
\put(183,95){$y\!=\!F_{3}(x)$}
\put(193,52){$y\!=\!\rme$}
\put(193,40){$y\!=\!F_{1}(x)$}
\end{picture}}
}\end{center}
\caption{Super-exponentials to base $b\!=\!\exp(1/\rme)$ versus real argument;
the asymptotic $y\!=\!\rme$; circles represent the data from \cite{walker123}.
\label{fig1}
\label{fige1e1}
~\iL{fig1e1fre}$\!\!\!\!$
}
\end{figure}

The two super-exponentials $F_{1}$ and $F_{3}$ along the real axis are shown in figure 
\ref{fig1}.
The circles represent the data from tables 1 and 3 by
\cite{walker123}. The behavior of these functions and their inverses
in the complex plane is shown in figure \ref{fig2}.

As in \cite{kouznetsov:sqrt2}, the subscript of the
super-exponential (here 1 or 3) indicates the value at 0 of the
chosen representative of the class of all super-exponentials obtained by argument shift $F(z\!+\!c)$, $c\in\mathbb{C}$ .
We often identify this whole class as one super-exponential. 
We consider two classes of super-exponentials
represented by $F_{1}$ with $F_1(0)\!=\!1$ and by $F_3$ with $F_3(0)\!=\!3$,
respectively. According to the definition, $F_{1}$ is a
tetrational. For the other (above unbounded) super-exponential, the smallest
integer from the range of values along the real axis is chosen as value at zero.

\newcommand \figaxe {
\put( -5,300){\sx{1}{$y$}}
\put( -9,289){\sx{1}{$14$}}
\put( -9,269){\sx{1}{$12$}}
\put( -9,249){\sx{1}{$10$}}
\put( -5,229){\sx{1}{$8$}}
\put( -5,209){\sx{1}{$6$}}
\put( -5,189){\sx{1}{$4$}}
\put( -5,169){\sx{1}{$2$}}
\put( -5, 149){\sx{1}{$0$}}
\put(-13,129){\sx{1}{$-2$}}
\put(-13,109){\sx{1}{$-4$}}
\put(-13,  89){\sx{1}{$-6$}}
\put(-13,  69){\sx{1}{$-8$}}
\put(-17,  49){\sx{1}{$-10$}}
\put(-17,  29){\sx{1}{$-12$}}
\put(-17,   9){\sx{1}{$-14$}}
\put(398,-6.5){\sx{1}{$x$}}
\put(377,-7){\sx{.9}{$28$}}
\put(357,-7){\sx{.9}{$26$}}
\put(337,-7){\sx{.9}{$24$}}
\put(317,-7){\sx{.9}{$22$}}
\put(297,-7){\sx{.9}{$20$}}
\put(277,-7){\sx{.9}{$18$}}
\put(257,-7){\sx{.9}{$16$}}
\put(237,-7){\sx{.9}{$14$}}
\put(217,-7){\sx{.9}{$12$}}
\put(197,-7){\sx{.9}{$10$}}
\put(179,-7){\sx{.9}{$8$}}
\put(159,-7){\sx{.9}{$6$}}
\put(139,-7){\sx{.9}{$4$}}
\put(119,-7){\sx{.9}{$2$}}
\put( 99,-7){\sx{.9}{$0$}}
\put( 73,-7){\sx{.9}{$-2$}}
\put( 53,-7){\sx{.9}{$-4$}}
\put( 33,-7){\sx{.9}{$-6$}}
\put( 13,-7){\sx{.9}{$-8$}}
}




\section{Four methods of calculating the regular iteration with multiplier 1}
In the theory of regular iteration (see e.g.\ \cite{szekeres:regular}
or \cite{kuczma:iterative}) there are several algorithms available to
compute the regular fractional/continuous iteration and the Abel
function of an analytic function at the fixed point 0. 
Functions $h$ with multiplier 1, e.g.\ $h'(0)=1$, are treated
differently from functions $h$ with $|h'(0)|\neq 0,1$.

In our case we have the base function $f(z)=\rme^{z/\rme}$ with fixed point
$\rme$ and $f'(\rme)=1$. As the whole theory of regular iteration assumes
the fixed point to be at 0, we move the fixed point to 0 via a
conjugation with the linear transformation $\tau$:
 Let $\tau(z)=\rme\,(z+1)$ then $\tau^{-1}(z)=z/\rme-1$ and 
\begin{align*}
  \tau^{-1}\circ f \circ \tau(z) = \rme^z - 1 &=: h(z)\\
  f &= \tau\circ h \circ \tau^{-1}
\end{align*}
The regular iterates $f^{[t]}$ at the fixed point $\rme$ are then given by
$f^{[t]} = \tau\circ h^{[t]}\circ \tau^{-1}$, where the regular iterates of
$h$ can be obtained in one of the later described ways. The regular
Abel function $\alpha$ of $h$ (up to an additive constant determined
by $u$) is defined by the inverse of $\sigma_u(t) = h^{[t]}(u)$. We call
this $\sigma_u$ the regular superfunction of $h$ with
$\sigma_u(0)=u$. 
\begin{align}
  \sigma_u(t) &:= h^{[t]}(u) & \alpha_u &:= \sigma_u^{-1} 
  & h^{[t]}(z) &= \sigma_u(t+\alpha_u(z))
\end{align}

The regular Abel function $A_u$ of $f$ at $\rme$
with $A_u(u)=0$ and the regular superfunction $F_u$ of $f$
at $\rme$ with $F_u(0)=u$ can be obtained by 
\begin{align}
  A_u &= \alpha_{\tau^{-1}(u)} \circ \tau^{-1} & F_u &= \tau\circ
  \sigma_{\tau^{-1}(u)}.
\end{align}

The classic limit formula of Lévy \cite{levy:fonctions} (see also
Kuczma \cite{kuczma:iterative} theorem 3.5.6) for the regular Abel
functions of $h$ with multiplier 1 is:
\begin{align}\label{eq:levy}
  \alpha_u(z)&=\lim_{n\to\infty}\frac{h^{[n]}(z) -
    h^{[n]}(u)}{h^{[n+1]}(u)-h^{[n]}(u)}\\
  A_u(z) &= \lim_{n\to\infty}\frac{f^{[n]}(z) - f^{[n]}(u)}{f^{[n+1]}(u)-f^{[n]}(u)}
\end{align}
One can verify that in our case of $h(z)\!=\!\rme^z\!-\!1$, Lévy's formula converges just too slowly; 
it is difficult to reach sufficient precision to make any camera-ready plot of the Abel function.
In table \ref{table:levy} we display
\begin{align*}
y_n=\frac{f^{[n]}(-1) - f^{[n]}(1)}{f^{[n+1]}(1)-f^{[n]}(1)}\to A_1(-1) 
\end{align*}

\newcommand{\sxsx}{\sx{.96}}
\begin{table}
\caption{Computing the Abel function with Lévy's formula.~ \iL{table:levy}}
\noindent
\sxsx{\begin{tabular}{cc|}
$n$& $y_n$\\
100& $-1.4560$~\\ 
101& $-1.4557$\\
102& $-1.4553$\\
103& $-1.4550$\\
104& $-1.4547$\\
105& $-1.4544$\\
106& $-1.4541$\\
107& $-1.4538$\\
108& $-1.4535$\\
109& $-1.4533$\\
\end{tabular}}
\sxsx{\begin{tabular}{cc|}
$n$& $y_n$\\
1,000& $-1.425788$~\\ 
1,001& $-1.425785$\\
1,002& $-1.425781$\\
1,003& $-1.425778$\\
1,004& $-1.425775$\\
1,005& $-1.425771$\\
1,006& $-1.425768$\\
1,007& $-1.425764$\\
1,008& $-1.425761$\\
1,009& $-1.425758$\\
\end{tabular}}
\sxsx{\begin{tabular}{cc|}
$n$& $y_n$\\
10,000& $-1.4226982$~\\ 
10,001& $-1.4226982$\\
10,002& $-1.4226981$\\
10,003& $-1.4226981$\\
10,004& $-1.4226981$\\
10,005& $-1.4226980$\\
10,006& $-1.4226980$\\
10,007& $-1.4226980$\\
10,008& $-1.4226979$\\
10,009& $-1.4226979$\\
\end{tabular}}
\sxsx{\begin{tabular}{cc}
$n$& $y_n$\\
100,000& $-1.42241848$\\ 
100,001& $-1.42241893$\\
100,002& $-1.42241823$\\
100,003& $-1.42241951$\\
100,004& $-1.42241880$\\
100,005& $-1.42241891$\\
100,006& $-1.42241937$\\
100,007& $-1.42241983$\\
100,008& $-1.42241913$\\
100,009& $-1.42241958$\\
\end{tabular}}
\end{table}

There is another interesting possibility to compute the regular
superfunction, which we call here {\em Newton limit formula} (probably
first mentioned by Écalle in \cite{ecalle:invariants}) because of its
similarity to the Newton binomial series of $x^t=(1+(x-1))^t$:
\begin{align}
  \label{eq:newton}
  \sigma_u(t) = h^{[t]}(u) &= \sum_{n=0}^\infty \binom{t}{n} \sum_{m=0}^n
  \binom{n}{m} (-1)^{n-m} h^{[m]}(u)\\
  F_u(t) = f^{[t]}(u) &= \sum_{n=0}^\infty \binom{t}{n} \sum_{m=0}^n
  \binom{n}{m} (-1)^{n-m} f^{[m]}(u).
\end{align}
However also this method has a depressing slow convergence, moreover
we need a bigger internal precision caused by the involved
summation. For example for 1000 summands and 2000 bits precision with $u=1$ and
$t=-1.4223536677333$ we get $F_u(t)\approx -0.9875$ while we would
expect a value very close to $-1$ 
(see table \ref{table:levy}).

Another formula to compute an Abel function of $\rme^x-1$ is given in
Walker's text \cite{walker123}. He computes an Abel function $g_1$
of $a(z)=1-\rme^{-z}$ and an Abel function $g_2$ of $h$ with a formula
which goes back to Fatou
\cite{fatou:equations}:
\begin{align}
  \label{eq:walker:o}
  g_1(z) &= \lim_{n\to\infty} - \frac{1}{3}\log(n) +
  \frac{2}{a^{[n]}(z)}-n, \quad z<0.
\end{align}
We derive the corresponding Abel function of $h$ by knowing that
$a(z)=-h(-z)$. 
\begin{align}
  \alpha^{(1)}_{\rm W}(z) = g_1(-z) &= \lim_{n\to\infty} -\frac{1}{3}\log(n) -
  \frac{2}{h^{[n]}(z)} - n, \quad z<0
~\iL{eq:walker}
\\
  \alpha^{(2)}_{\rm W}(z) =g_2(z) &= \lim_{n\to\infty} - \frac{1}{3}\log(n) -
  \frac{2}{h^{[-n]}(z)}+n, \quad z\ge 0
~\iL{eq:walker2}
\end{align}
The convergence of this formula is better than that of Lévy
but still rather slow (which Walker notices too and
that's why he introduces a slightly accelerated version which we omit
here for brevity). To have an impression of the convergence of
Fatou's/Walker's formula, we display
\begin{align*}
y_n&:=-\frac{2}{h^{[n]}(-\rme^{-1}-1)}+\frac{2}{h^{[n]}(-1)}-1\\
  &\longrightarrow\;\alpha^{(1)}_{\rm W}(\tau^{-1}(-1))-\alpha^{(1)}_{\rm W}(\tau^{-1}(0))-1
= A_1(-1)
\end{align*}
in table \ref{table:fatou}.
\begin{table}
\caption{Computing the Abel function with Fatou's formula.~\iL{table:fatou}}
\begin{tabular}{cc|}
$n$& $a_n$\\
1,000   & $-1.4224939$\\
1,001   & $-1.4224938$\\
1,002   & $-1.4224936$\\
1,003   & $-1.4224935$\\
1,004   & $-1.4224934$\\
1,005   & $-1.4224932$\\
\end{tabular}
~
\begin{tabular}{cc|}
$n$& $a_n$\\
10,000  & $-1.422367740$\\
10,001  & $-1.422367738$\\
10,002  & $-1.422367737$\\
10,003  & $-1.422367736$\\
10,004  & $-1.422367734$\\
10,005  & $-1.422367733$\\
\end{tabular}
~
\begin{tabular}{cc}
$n$& $a_n$\\
100,000 & $-1.42235507550$\\
100,001 & $-1.42235507549$\\
100,002 & $-1.42235507548$\\
100,003 & $-1.42235507546$\\
100,004 & $-1.42235507545$\\
100,005 & $-1.42235507543$\\
\end{tabular}
\end{table}

Before we give the fourth method and showing that Walker's formula is
equivalent to it, we start with some formal background
about regular iteration.
\begin{definition}[regular iteration]
  For every formal powerseries 
  \begin{align}
    h(z)=z+\sum_{n=m}^\infty h_n z^n,\quad m\ge 2, h_m\neq 0
  \end{align}
  and each $t\in\C$ there is exactly one
  formal powerseries $h^{[t]}(z)=z+\sum_{n=m}^\infty {h^{[t]}}_n z^n$,
  such that ${h^{[t]}}_m = t\cdot h_m$ and $h^{[t]}\circ h = h\circ
  h^{[t]}$. We call $h^{[t]}$ the regular iteration of
  $h$. It satisfies $h^{[1]}=h$ and $h^{[s+t]}=h^{[s]}\circ h^{[t]}$
  and is given by the formula:
  \begin{align}
    {h^{[t]}}_N &= \sum_{n=0}^{N-1} \binom{t}{n} \sum_{m=0}^n \binom{n}{m}
    (-1)^{n-m} {h^{[m]}}_N \\
    \label{eq:jabotinsky}
    &= \sum_{m=0}^{N-1} (-1)^{N-1-m}
    \binom{t}{m}\binom{t-1-m}{N-1-i} {h^{[m]}}_N
  \end{align}
  where \eqref{eq:jabotinsky} can already be found as formula (2.19) in
  \cite{Jabotinksy:AnalyticIteration}.

  The formal powerseries $h^{[t]}$ is not necessarily convergent even
  if $h$ is. We call a
  function which has the powerseries $h^{[t]}$ as asymptotic expansion
  at 0 a regular iteration of $h$.

  If $z\mapsto h^{[t]}(z)$ is an analytic function in some domain, we
  call the function 
  $\sigma(t)=h^{[t]}(z_0)$ 
  a regular superfunction
  of $h$ for any $z_0$ in the domain, and we call its inverse a
  regular Abel function of $h$. Usually we identify Abel functions
  that only differ by a constant and we identify superfunctions that
  are translations of each other (i.e. $x\mapsto\sigma(x+c)$
  is identified with $\sigma$).
\end{definition}

\begin{definition}
  Let $h$ be a formal powerseries of the form
  $h(x)=x+\sum_{k=m}^\infty h_k x^k$, $h_m\neq 0$. Its iterative
  logarithm is the unique formal powerseries $j$ of form
  $j(x)=\sum_{k=m}^\infty j_k x^k$ with
  $j_m=h_m$ that satisfies the Julia equation
  \begin{align}\label{eq:julia_equation}
    j\circ h = h' \cdot j.
  \end{align}
\end{definition}
One obtains the Julia equation when differentiating the Abel equation
and then substituting $\alpha' = 1/j$. For reference we give the first
few coefficients of the iterative logarithm $j$ of $h(x)=\rme^x-1$:
\begin{align*}
j(x)=\frac{1}{2}x^2-\frac{1}{12}x^3+\frac{1}{48}x^4-\frac{1}{180}x^5+\frac{11}{8640}x^6-\frac{1}{6720}x^7+\dots
\end{align*}
From this iterative logarithm one can get a description of the regular
Abel function by $\alpha = \int \frac{1}{j}$.
\begin{align*}
\alpha'(x)=\frac{1}{j(x)}=2x^{-2}+
\frac{1}{3}x^{-1}-\frac{1}{36}+\frac{1}{270}x+\frac{1}{2592}x^2-\frac{71}{108864}x^3+\dots
\end{align*}
If we integrate this to get $\alpha$ the term $x^{-1}$ becomes
$\log|x|$ for real $x$ and $\log(\pm x)$ for complex values of $x$; the choice of the sign determines the branch of the resulting function (that unavoidably has the cutline). This gives the expansion 
\begin{align}\label{eq:e1_abel}
\alpha(x)=\underbrace{-2x^{-1}+\frac{1}{3}\log(\pm x)}_{s(x)}
\underbrace{-\frac{1}{36}x+ \frac{1}{540}x^2+
\frac{1}{7776}x^3 -\frac{71}{435456}x^4+
\dots}_{v(x)}
\end{align}
This formula can not be used to get {\em arbitrary} precision,
because $v(x)$ is not convergent as we show now; however 
suitable truncation of the divergent series can be used to obtain
a {\em certain} precision.

\begin{preliminary}[Baker 1958 \cite{Baker:Zusammensetzungen} Satz 17]\label{thm:baker}
The regular iteration $h^{[t]}$ of $h(x)=\rme^x-1$ has
non-zero convergence radius exactly if $t$ is an integer.
\end{preliminary}


\begin{preliminary}[Écalle 1975
\cite{ecalle:theorie_iterative}
]\label{thm:iter:itlog}
Let $h$ be a formal powerseries with multiplier 1. Its regular
iteration powerseries $h^{[t]}$ has a positive radius of convergence
for all $t\in\C$ if and only if its iterative logarithm has a
positive radius of convergence.
\end{preliminary}

\begin{theorem}
The formal powerseries $v(z)$ in \eqref{eq:e1_abel} has 0 convergence
radius.
\end{theorem}
\begin{proof}
  Suppose that $v(z)$ has non-zero radius of convergence. Then also 
  $v'(z)=\alpha'(z)-s'(z)$ has non-zero radius of convergence. Then
  $z^2v(z)$ has non-zero radius of convergence and then $z^2 s'(z) +
  z^2 v(z)$ is a powerseries with non-zero radius of convergence with
  non-zero zeroth coefficient. Then
  \begin{align*}
    j(z)=\frac{1}{\alpha'(z)} = \frac{z^2}{z^2s'(z)+z^2v(z)}
  \end{align*}
  has non-zero radius of convergence. Then by theorem
  \ref{thm:iter:itlog} the regular iteration $h^{[t]}$ of
  $h(x)=\rme^x-1$ has non-zero radius of convergence for all $t$. This
  is in contradiction to theorem \ref{thm:baker}. Hence, the series 
diverges.
\end{proof}

Nonetheless we can use the formula \eqref{eq:e1_abel} in a different way
to calculate the regular Abel function of $h(x)=\rme^x\!-\!1$. If we truncate
$v$ to $N$ summands, denoted by $v_N$ and $\alpha_N:=s\!+\!v_N$, then 
Écalle showed (see \cite{ecalle:invariants}, p.\ 78 ff., 95 ff.) 
that there are $2 (m\!-\!1)$ different regular Abel functions
$\alpha^{(j)}$, $1\le j \le 2 (m-1)$ ($\alpha^{(j)}$ is defined on
the $j$-th petal --- each petal touching the fixed point 0 --- of the now called Leau-Fatou flower see
\cite{Milnor2006}) given by:
\begin{align}
\label{eq:ecalle_iterative}\alpha^{(j)}(z)=\lim_{n\to\infty} \alpha_N^{(j)}(f^{[(-1)^{j+1}
  n]}(z))-(-1)^{j+1}n, \quad N\in\N
\end{align}
where $\alpha_N^{(j)}(z)$ is $\alpha_N(z)$ with the logarithmic term
$\log|z|$ (in \eqref{eq:e1_abel}) replaced by $\log\left(z e^{-(\theta_0+j\frac{\pi}{m-1})\rmi}\right)$ and
$\theta_0$ is the unique number in the interval
$(-\frac{\pi}{m-1},\frac{\pi}{m-1}]$ such that $h_m= \abs{h_m}
e^{-\rmi (m-1) \theta_0}$.

This applied to $h(z)=\rme^z-1$ where $m=2$, $h_m=1/2$, $\theta_0=0$,
 we get the two regular Abel functions
\begin{align}
  \label{eq:ecalle1}
  \alpha^{(1)}(z) &= \lim_{n\to\infty} \frac{1}{3}\log(-h^{[n]}(z)) -
  \frac{2}{h^{[n]}(z)} + v_N(h^{[n]}(z)) - n, \quad \Re(z)<0\\
  \label{eq:ecalle2}\alpha^{(2)}(z) &= \lim_{n\to\infty} \frac{1}{3}\log(h^{[-n]}(z)) -
  \frac{2}{h^{[-n]}(z)} + v_N(h^{[-n]}(z))+ n, \quad \Re(z)>0\\
  \label{eq:Ecalle1} A^{(1)}(z) &= \alpha^{(1)}\left(\frac{z}{\rme} - 1\right), \Re(z)<\rme\\
  \label{eq:Ecalle2} A^{(2)}(z) &= \alpha^{(2)}\left(\frac{z}{\rme} - 1\right), \Re(z)>\rme
\end{align}

From these 4 methods we know that Lévy's formula \eqref{eq:levy}, the
Newton formula \eqref{eq:newton}
and Écalle's method \eqref{eq:ecalle1} and \eqref{eq:ecalle2} calculate the {\em regular} iteration/Abel function. We
show in the last part of this section that Walker's formula is equal
to Écalle's formula and hence also computes the regular Abel function.

The application of theorem 1.3.5 in \cite{kuczma:iterative}
gives the following:
\begin{preliminary}[Thron 1960 \cite{thron:sequences} Theorem 3.1.]\label{pre:thron}
Let $h$ be analytic at 0 with powerseries
expansion of the following form
\begin{align*}h(x)=x+h_{m} x^{m} + h_{m+1}x^{m+1} + \dots,\quad
  h_{m}<0,m\ge 2\end{align*}
then 
\begin{align*}
  \lim_{n\to\infty} n^{1/(m-1)} h^{[n]}(x) = (-h_{m} (m-1))^{-1/(m-1)}.
\end{align*}
\end{preliminary}
Now, about the  functions $\alpha^{(1)}_{\rm W}$ and $\alpha^{(2)}_{\rm W}$ constructed by Walker, we have the following theorem:
\begin{theorem}
Functions $\alpha^{(1)}_{\rm W}$ and $\alpha^{(2)}_{\rm W}$
given in \eqref{eq:walker} and \eqref{eq:walker2} 
are the two regular Abel functions of $x\mapsto \rme^x\!-\!1$.
\end{theorem}
\begin{proof}
We show that the difference of Walker's and Écalle's limit
formulas is a constant. The differences are:
\begin{align*}
  \delta_1(z) = \lim_{n\to\infty}\frac{1}{3}\log(-h^{[n]}(z)) +
  \frac{1}{3}\log(n) = \lim_{n\to\infty} \frac{1}{3}\log\left(-n h^{[n]}(z)
  \right), \quad z<0\\
  \delta_2(z) = \lim_{n\to\infty}\frac{1}{3}\log(h^{[-n]}(z)) +
  \frac{1}{3}\log(n) = \lim_{n\to\infty} \frac{1}{3}\log\left(n
    h^{[-n]}(z)\right),\quad z>0
\end{align*}
As $x\mapsto -h(-x)$ and $x\mapsto h^{-1}(x)$ for $h(x)=\rme^x-1$ is of the form required by 
preliminary \ref{pre:thron}
with $m-1=1$ we see that each of $n h^{[n]}(z)$ and $n h^{[-n]}(z)$ converges to a constant
independent on $z$.
\end{proof}


\section{A new expansion of the super-exponentials}

This section describes an evaluation of the two super-exponentials to base $b=\exp(1/\rme)$.
The given expansion is fast and precise; it allows to plot the complex maps of these functions in real time.
These maps are shown in figure \ref{figmapf}. 
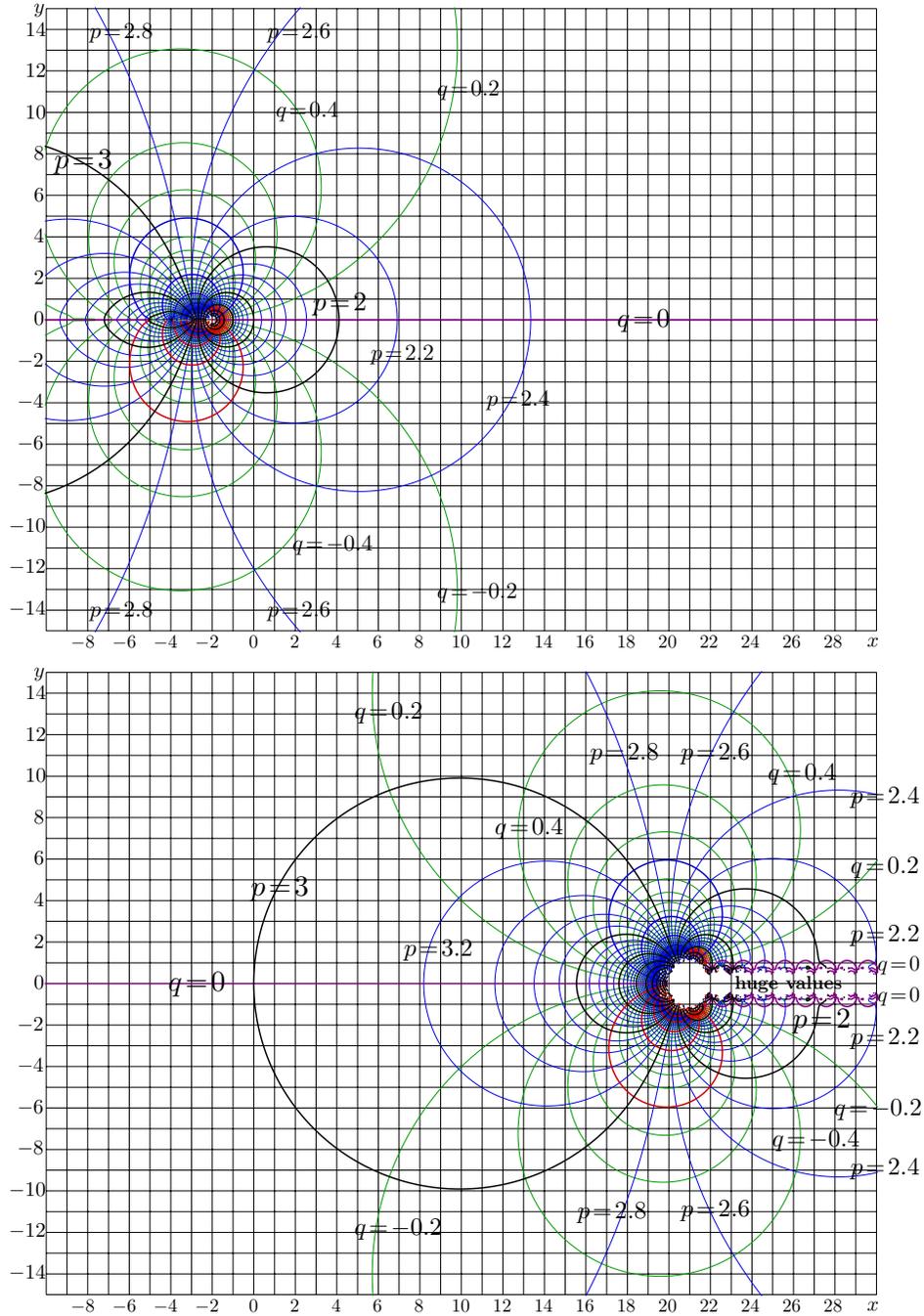
\begin{figure}
\begin{center}
 \sx{.8}{\begin{picture}(415,310) 
\put(0,0){\ing{e1etf}} 
 \figaxe
 \put(  5,225){\sx{1.4}{$p\!=\!3$}}
 \put( 22,288){\sx{1.1}{$p\!=\!2.8$}}  \put(108,288){\sx{1.1}{$p\!=\!2.6$}}
 \put( 22,   8){\sx{1.1}{$p\!=\!2.8$}}  \put(108,    8){\sx{1.1}{$p\!=\!2.6$}}
 \put(112,250){\sx{1.1}{$q\!=\!0.4$}}
 \put(120,  40){\sx{1.1}{$q\!=\!-0.4$}}
 \put(214,110){\sx{1.1}{$p\!=\!2.4$}}
 \put(190,260){\sx{1.1}{$q\!=\!0.2$}}
 \put(130,156){\sx{1.3}{$p\!=\!2$}}
 \put(277, 148.4){\sx{1.3}{$q\!=\!0$}}
 \put(158,132){\sx{1.1}{$p\!=\!2.2$}}
 \put(190, 17){\sx{1.1}{$q\!=\!-0.2$}}
 \end{picture}}

 \sx{.8}{\begin{picture}(415,320) 
 \put(0,0){\ing{e1egf} }  
 \figaxe
 \put(100,194){\sx{1.4}{$p\!=\!3$}}
 \put(150,280){\sx{1.2}{$q\!=\!0.2$}}
 \put(150, 30){\sx{1.2}{$q\!=\!-0.2$}}
 \put(218,224){\sx{1.2}{$q\!=\!0.4$}}
 \put(264,260){\sx{1.2}{$p\!=\!2.8$}}
 \put(258, 38){\sx{1.2}{$p\!=\!2.8$}}
 \put(308,260){\sx{1.2}{$p\!=\!2.6$}}
 \put(308, 39){\sx{1.2}{$p\!=\!2.6$}}
 \put(350,250){\sx{1.2}{$q\!=\!0.4$}}
 \put(390,239){\sx{1.2}{$p\!=\!2.4$}}
 \put(390, 59){\sx{1.2}{$p\!=\!2.4$}}
 \put(390,205){\sx{1.2}{$q\!=\!0.2$}}
 \put(382, 88){\sx{1.2}{$q\!=\!-0.2$}}
 \put(352, 73){\sx{1.2}{$q\!=\!-0.4$}}
 \put(390,172){\sx{1.2}{$p\!=\!2.2$}}
 \put(390,122){\sx{1.2}{$p\!=\!2.2$}}
 \put(403,158){\sx{1.04}{$q\!=\!0$}}
 \put(403,143){\sx{1.04}{$q\!=\!0$}}
 \put(362,130){\sx{1.4}{$p\!=\!2$}}
 \put(174,165){\sx{1.2}{$p\!=\!3.2$}}
 \put( 60,148){\sx{1.4}{$q\!=\!0$}}
 \put(334,149){\sx{.9}{\bf huge values}}
 \end{picture}}

\end{center}
\caption{
Map of $f\!=\!F_{1}(z)$, top, and $f\!=\!F_{3}(z)$, bottom, in the plane $z\!=\!x\!+\!\rmi y$.
Levels
$p\!=\!\Re(f)\!=\!\rm const$ and 
$q\!=\!\Im(f)\!=\!\rm const$ are shown; thick lines correspond to the integer values.
\label{fig2}
\iL{figmapf}
}
\end{figure}

The base function $h(z)=\exp_{b}(z)=\exp(z/\rme)$ has the only fixed point $z=\rme$.
The super-exponential  is expected to approach this point asymptotically.
Consider the expansion of the super-exponential $f$ in the following form:
\be
\tilde F(z)=\rme\cdot\left(1-\frac{2}{z}\left(
	1+\sum_{m=1}^{M} \frac{P_{m}\big(-\ln(\pm z) \big)}{(3z)^m}  
	+\mathcal{O}\!\left(\frac{|\ln(z)|^{m+1}}{z^{m+1}}\right)
\right) \right)
\eL{fo} 
where
\be
P_{m}(t)=\sum_{n=0}^{m} c_{n,m} t^{n}
\eL{P}
The substitution of \eqref{f} into equation
\be
F(z\!+\!1)=\exp(F(z)/\rme)
\eL{Fexpe}
and the asymptotic analysis with small parameter $|1/z|$ determines
the coefficients $c$ in the polynomials \eqref{P}. 
In particular, 
\be
P_{1}(t)&=&t \\
P_{2}(t)&=&t^{2}+t+1/2 \\
P_{3}(t)&=&t^{3}+\frac{ 5}{ 2}t^{2}+\frac{  5}{2}t    +\frac{ 7}{10} \\
P_{4}(t)&=&t^{4}+\frac{13}{ 3}t^{3}+\frac{ 45}{6}t^{2}+\frac{53}{10}t    +\frac{ 67}{60} \\
P_{5}(t)&=&t^{5}+\frac{77}{12}t^{4}+\frac{101}{6}t^{3}+\frac{83}{ 4}t^{2}+\frac{653}{60}t+\frac{2701}{1680}
\eL{Pc}
The evaluation with 9 polynomials $P$ gives an approximation of $f(z)$ with 15 decimal digits at $\Re(z)>4$.
For small values of $z$, the iterations of formula  
\be
F(z)=\ln(F(z\!+\!1))\!~ \rme
\eL{fzL}
can be used. With {\tt complex$<$double$>$} precision, the resulting approximation returns of order of 14 correct decimal digits in the whole complex plane, except the singularities.

For the tetrational we choose the negative sign inside the
logarithm $t=-\ln(-z)$. Then
\be
F_{1}(z)=\tilde F(z+x_{1})
\eL{x1}
where $x_{1}\approx 2.798248154231454$ is the solution of the equation $f(x_{1})\!=\!1$.
For real values of the argument, this function is shown at the bottom of figure \ref{fig1}.
The complex map of this function is shown at the top of figure \ref{fig2}.

The same expressions \eqref{Pc} can be used also for the above
unbounded super-exponential with $t=-\ln(z)$.
Then, the expression
\be
F(z)=\exp(F(z\!-\!1)/\rme)
\eL{fzE}
allows the evaluation of the above unbounded super-exponential at small $|z|$.
The specific super-exponential $F_{3}$ can be expressed as
\be
F_{3}(z)=\tilde F(z+x_{3})
\eL{x3}
where $x_{3}\approx -20.28740458994004$ is solution of equation $\tilde F(x_3)=3$.
Function $F_{3}$ by \eqref{x3},\eqref{fzE},\eqref{f} is shown at the top of figure 
\ref{fig1}
for real argument and in the bottom picture of figure 
\ref{fig2}
for the complex values of its argument.

Due to the leading term in the asymptotic representation, at large values of $|z|$ (except the vicinity of the real axis), both functions $F_{3}(z)$ and $F_{1}(z)$ behave similar to the function  $z \mapsto 1/z$.


$F_3$ is entire, and shows fast growth along the real axis. At large values of $|z|$, the function $F_1(z)$ approaches value $\rme$. Function $F_3(z)$ approaches value 
$\rm e$ for $|z|\rightarrow \infty$ except on the positive direction of the real axis; in this direction, this function shows ``faster than any exponential'' growth.

\section{Numerics and behavior of the two super-logarithms}

\begin{figure}
\begin{center}
 \sx{.8}{\begin{picture}(415,310) 
 \put(0,0){\ing{e1eti} } 
 \figaxe
 \put(10,233){\sx{1.8}{$q\!=\!0$}} 
 \put(142,245){\sx{1.2}{$p\!=\!-3$}}
 \put(142,225){\sx{1.2}{$p\!=\!-3$}}
 \put(10,192){\sx{1.8}{$p\!=\!-2$}}
 \put(10,147){\sx{1.8}{$q\!=\!0$}}
 \put(10,108){\sx{1.8}{$p\!=\!-2$}}
 \put(10, 61){\sx{1.8}{$q\!=\!0$}} 
 \put(142, 72){\sx{1.2}{$p\!=\!-3$}}
 \put(142, 52){\sx{1.2}{$p\!=\!-3$}}
 \put(229,290){\sx{1.6}{\bf complicated structure}} 
 \put(250,230){\sx{1.8}{$A_1(x\!+\!\rmi y) \approx -3$}} 
 \put(229,148){\sx{1.6}{\bf complicated structure}} 
 \put(250,60){\sx{1.8}{$A_1(x\!+\!\rmi y) \approx -3$}} 
 \put(229,5){\sx{1.6}{\bf complicated structure}} 
 \end{picture}}
 \sx{.8}{\begin{picture}(415,320) 
 \put(0,0){\ing{e1egi} }  
 \figaxe
 \put(  4,220){\sx{1.4}{$p\!=\!19$}} 
 \put(352,276){\sx{1.4}{$q\!=\!0.2$}}
 \put(350, 26){\sx{1.4}{$q\!=\!-0.2$}}
 \put(350,186){\sx{1.4}{$p\!=\!18.4$}}
 \put(270,176){\sx{1.4}{$p\!=\!18.2$}}
 \put(233,158){\sx{1.4}{$p\!=\!18$}}
 \put(  4, 74){\sx{1.4}{$p\!=\!19$}} 
 \put(20,186){\sx{1.4}{$q\!=\!1$}}
 \put(20,110){\sx{1.4}{$q\!=\!-1$}}
 \put(20,148.4){\sx{1.4}{\bf cut }} 
 \put(377,148.4){\sx{1.3}{$q\!=\!0$}}
 \end{picture}}
\end{center}
\caption{
The Abel functions  $f\!=\!A_1(z)$ , top, and $f\!=\!A_3(z)$, bottom, in the same notations as in figure \ref{fig2}.
\iL{fig3}
\iL{figi} $\!\!\!\!\!\!\!\!$
}
\end{figure}
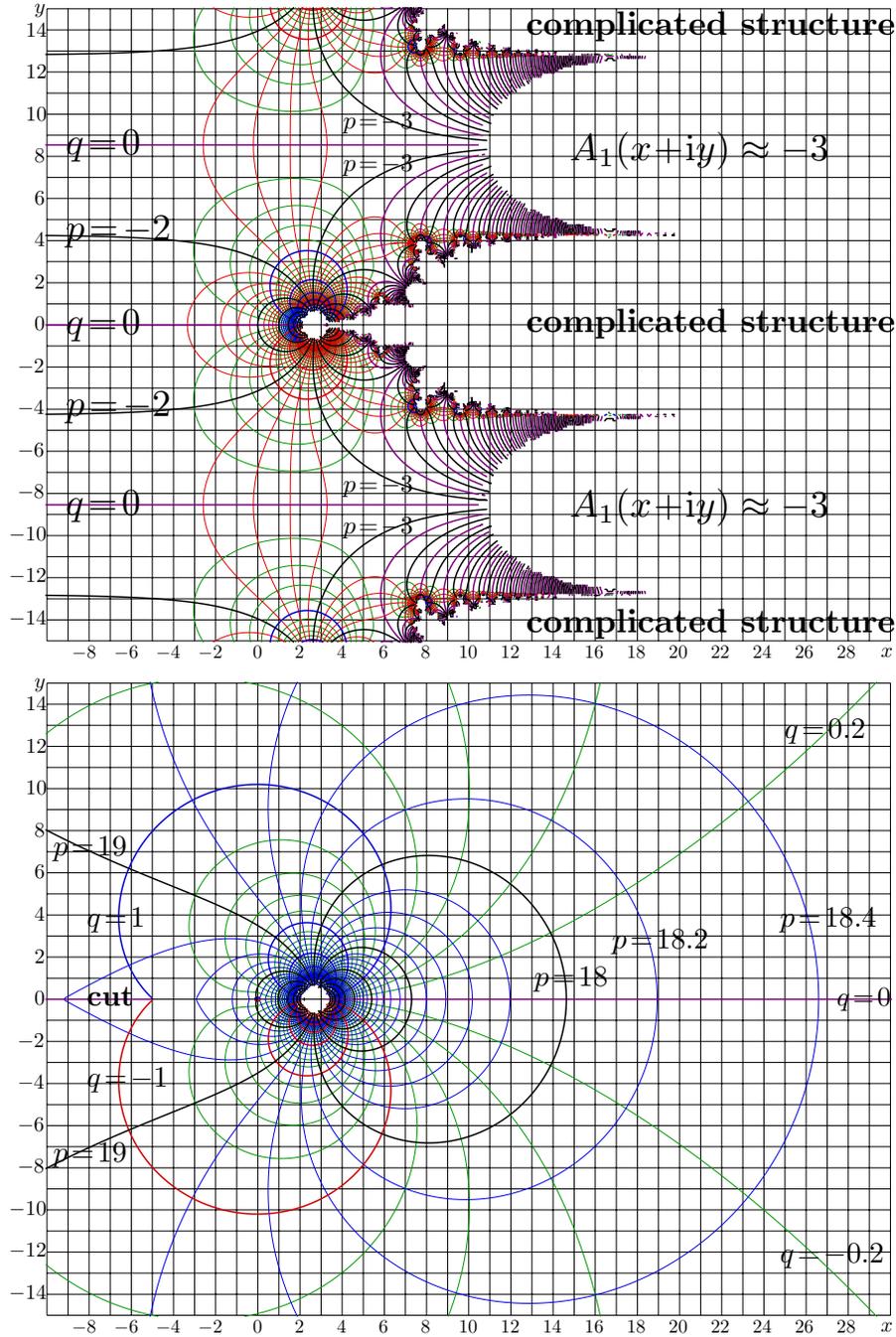

For the evaluation of the super-logarithms, the expression (\ref{eq:e1_abel}) is used. The truncation of the series $v$ keeping the term of the 15th power was used to build-up an approximation of $A^{(1)}$ that returns at least 15 decimal digits for $|z/\rme-1|<\frac{1}{2}$.
\be
A^{(1)}(z)&=& \alpha(-\zeta) \approx \frac{\ln(\zeta)}{3}+\frac{2}{\zeta}+ \sum_{n=1}^{15} c_n \zeta^n 
\label{expanA1app}
\ee
where $\zeta=(\rme-z)/\rme$.
For larger values, the representation
\be
A^{(1)}(z)&=A^{(1)}\big(\exp(z/\rme)\big)+1\iL{gz1}
\label{equaA1}
\ee
is iteratively used.
This allows to extend the approximation to a wide domain keeping of order of 14
correct decimal digits. Then
\be
A_1(z) &=&A^{(1)}(z)-A^{(1)}(1) ~\approx~ A^{(1)}(z) - 3.029297214418 ~
\eL{A3.02929}
is the regular Abel function with the additive constant chosen such
that $A_1(1)\!=\!0$. This function is shown in the top of figure \ref{figi}.

The function $A_1$ is periodic; the period is $T_{1}=2\pi \rme\, \rmi \approx 
17.079468445347134131\!~\rmi$. For real values $z\!>\!\rme$, the representation diverges, 
indicating a natural way to place the cut of the range of holomorphism. In vicinity of this cut, the Abel-function shows  complicated, fractal-like behavior: the self-similar structures reproduce long the range with high density of levels of constant real or imaginary part of $A_1$. 

At the left hand side of the picture, $A_1(z)$ approaches its asymptotic value $-2$ as $\Re(z)\rightarrow -\infty$ . Along the strips in vicinity $\Im(z)=\Im(T_1 + \pi \rmi (2n\!+\!1))$, $n\!\in\!~$integers, as $\Re(z)\rightarrow +\infty$, this function approaches its another limiting value:
$A_1(z) \rightarrow -3$. The transfer from the asymptotic value $-2$ to the asymptotic value $-3$ corresponds to the transition from the singularity at 
$z\!=\!-2$ to the singularity    
$z\!=\!-3$ of function $F_1(z)$ in the top picture of figure \ref{figmapf}. 

The second super-logarithm $A^{(2)}$ has also good approximation for small
values of $|\zeta|$ for the same $\zeta=(\rme\!-\!z)/\rme$:
\be
A^{(2)}(z)& \approx & \frac{\ln(-\zeta)}{3} + \frac{2}{\zeta} + \sum_{n=1}^{16} c_n \zeta^n
\ee
with the same coefficients $c$, as in the case of $A^{(1)}$.
(The only difference is the opposite sign in the argument of the logarithm.)
For large values, an extension to a wide range
in the complex plane can be similarly realized with 
\be
A^{(2)}(z)&=&A^{(2)}\!\big(\log(z)~\rme\big)-1
~
\eL{gz2}
The resulting function 
\be
A_3(z)=A^{(2)}(z)-A^{(2)}(3) \approx A^{(2)}(z) + 20.0563555297533789
\ee
 is plotted in the right bottom part of figure \ref{figi}. Function $A_3(z)$ is not periodic, and has cut from the branch-point $z\!=\!\rme$ to the negative direction of the real axis. In the positive direction of the real axis, it grows to infinity, and this grow is very slow (slower than any finite combination of logarithms).
 
The asymptotic representation for the Abel functions $A_1$ and $A_3$ can be inverted, using various combinations of $A-\rme$ and  
$\ln\!\big(\!\pm(A-\rme)\big)$ as a small parameter. Different small parameter allows different representations for the super-exponentials $F_1$ and $F$;
it seems many of them give comparable speed and comparable prevision;
at least they do not add much errors to the rounding errors at the
$\rm complex\!<\!doube\!>$ implementation. The asymptotics of the previous section seems to be the fastest, although the careful comparison of efficiency of various asymptotic formulas may be subject for the future investigation.
 
For the plotting of figure \ref{figi}, the algorithms for $A_1$ and $A_3$ were implemented in C++ with $\rm complex\!<\!doube\!>$ arithmetics.
In order to verify the consistency of these algorithms to those for $F_1$ and $F_3$, the following agreements are considered:
\be
D_{\rm 1AF}(z)= \lg \left| \frac
{A_1(F_1(z))+z}
{A_1(F_1(z))-z} \right|
~ &,&~
D_{\rm 1FA}(z)= \lg \left| \frac
{F_1(A_1(z))+z}
{F_1(A_1(z))-z} \right|
\iL{D1}\\
\rule{0pt}{3pt}\nonumber\\
D_{\rm 3AF}(z)= \lg \left| \frac
{A_3(F_3(z))+z}
{A_3(F_3(z))-z} \right|
~ &,&~
D_{\rm 3FA}(z)= \lg \left| \frac
{F_3(A_3(z))+z}
{F_3(A_3(z))-z} \right|
\eL{D3}
Two of them, namely, $D\!=\!D_{\rm 1AF}(z)$ and $D\!=\!D_{\rm 3FA}(z)$ are shown in figure \ref{figd} with contours $D\!=\!\rm const$. As for $D_{\rm 1FA}(z)$ and $D_{\rm 3AF}(z)$, they remain of order of 14 in the whole range of such a picture, so, they are not presented here.

\begin{figure}
\begin{center}
\sx{.8}{\begin{picture}(415,310) 
\put(0,0){\ing{e1etfi}} 
 \figaxe
 \put( 80,262){\sx{1.5}{$D_{\rm 1FA}\!<\!1$}}
 \put( 10,170){\sx{1.5}{$12\!<\!D_{\rm 1FA}\!<\!14$}}	
 \put(202, 128){\sx{1.5}{$D_{\rm 1FA}\!<\!1$}}
 \put(126,160){\sx{1.4}{$15$}}
 \put( 80, 42){\sx{1.6}{$D_{\rm 1FA}\!<\!1$}}
 \put(182,229){\sx{1.1}{11}}
 \put(182,66){\sx{1.1}{11}}
 \end{picture}}
 \sx{.8}{\begin{picture}(415,320) 
\put(0,0){\ing{e1egif}} 
 \figaxe
 \put( 90,242){\sx{1.5}{$D_{\rm 3AF}\!>\!14$}}
 \put(355,146){\textcolor{white}{\rule{30pt}{12pt}}  }
 \put(356,148){\sx{1.2}{$D\!<\!1$}}
 \end{picture}}
\end{center}
\caption{
The agreements 
 $D=D_{\rm 1FA}(z)$ and 
 $D=D_{\rm 3AF}(z)$ and 
in the complex plane $z\!=\!x\!+\!\rmi y$; level $D\!=\!1$ is shown with thick lines;
the integer values are shown with thin lines;
levels $D\!=\!2$ is shown with black thick lines,
levels $D\!=\!1,2,10,12,14$ are seen;
symbols ``15'' and ``11'' indicate the ranges where $D\!>\!14$ and 
$10\!>\!D\!>\!12$.
\iL{figd} $\!\!\!\!\!\!\!\!$
}
\end{figure}
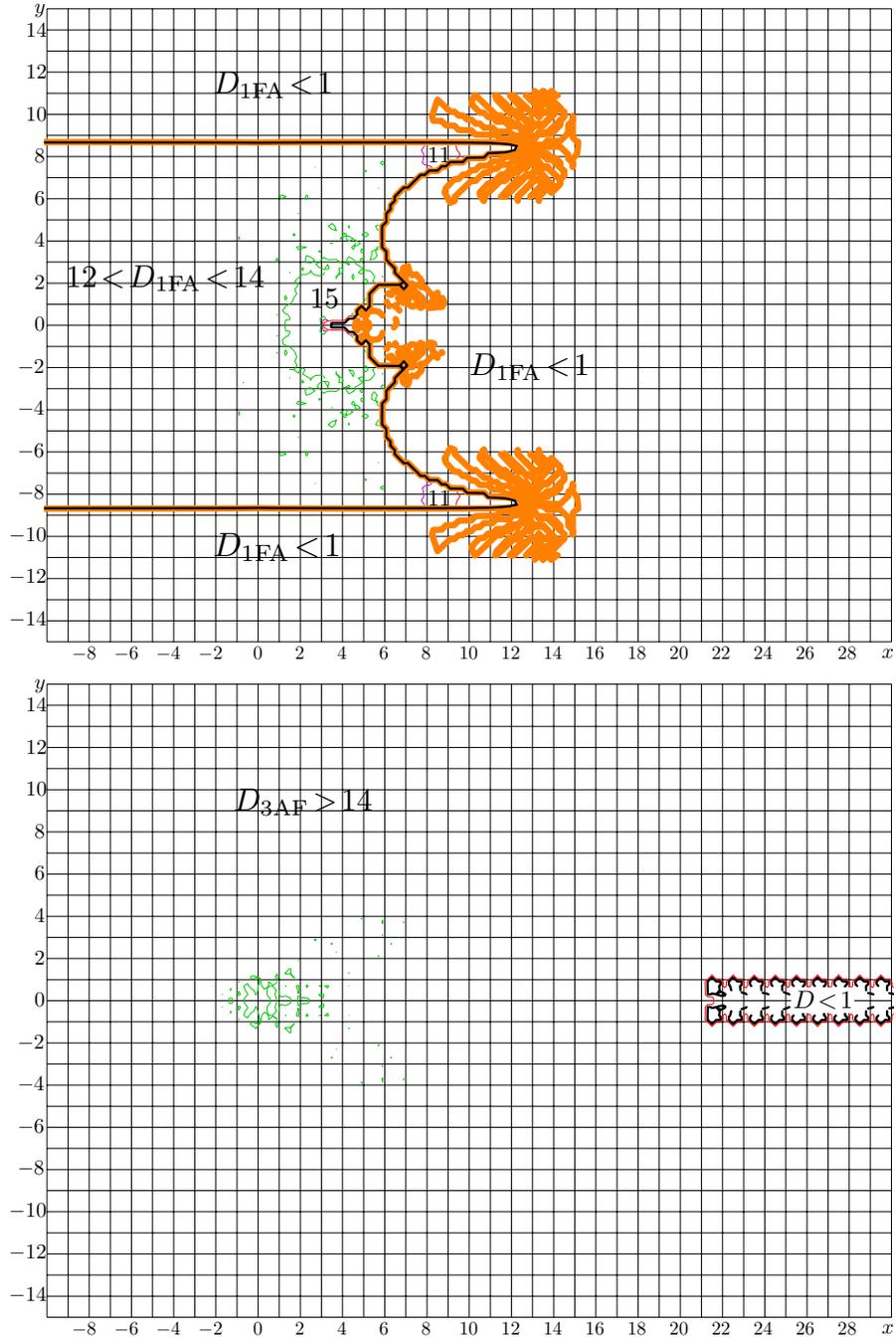

In figure \ref{figd}, symbol ``15'' indicates the region where $D>14$, and symbol ``11" 
indicates the ranges where $10<D<12$.

The top picture in figure \ref{figd} shows, that in the central part, the implementations of functions $F_1$ and $A_1=F_1^{-1}$ are consistent within at least 14 decimal digits.
In the right hand side, the branches of functions $F_{1}$ and $A_1$ do not match, and the agreement is poor.

The good agreement indicates, that the algorithms above work close to
the best precision achievable with the {\tt complex$<$double$>$} variables. However, the range of validity of relation $F(F^{-1}(z))=z$ is limited by the cut lines of the functions $F$ and $F^{-1}$.
In such a way, the numerical tests confirm the efficient C++ implementation of the super-exponentials  $F_1$ and $F_3$ and the corresponding Abel-exponentials  $A_1$ and $A_3$.


\section{Non-integer iteration}

Each of the pairs ($F_{1},A_{1}$) and ($F_{3},A_{3}$)
can be used to construct the regular iteration of the exponential to base $b\!=\!\exp(1/\rme)$:
\be
\exp_{b,1}^{[c]}(z)=F_1\!\big(c+A_{1}(z)\big) ~& \iL{q1}\\
\exp_{b,3}^{[c]}(z)=F_3\!\big(c+A_{3}(z)\big) ~& \iL{q3}
\ee
These functions are shown in figure \ref{figq} for $c\!=\!1/2$.
For comparison, the function $y\!=\!\exp_{b}(z)$ is plotted with a
thin curve. Visually, the thick solid curve looks like a continuation of the dashed curve. 
However, analytic continuation is not possible because $z\!=\!\rme$ is
a branch point of $A_{1}(z)$. Similar visual effects are discussed in \cite{kouznetsov:sqrt2} for the case $b\!=\!\sqrt{2}$.
\begin{figure}
\begin{center}{
\sx{1.8}{\begin{picture}(190,126)
\put(0,0){\ing{fige1eq}}
\put( 83,119){\sx{.8}{$y$}}
\put( 83,99.5){\sx{.7}{$8$}}
\put( 83,79.5){\sx{.7}{$6$}}
\put( 83,59.5){\sx{.7}{$4$}}
\put( 78,47.5){\sx{.7}{$\rme$}}
\put( 83,39.5){\sx{.7}{$2$}}
\put( 83,19.5){\sx{.7}{$0$}}
\put( -4.4,23){\sx{.7}{$-8$}}
\put( 15.6,23){\sx{.7}{$-6$}}
\put( 35.6,23){\sx{.7}{$-4$}}
\put( 55.6,23){\sx{.7}{$-2$}}
\put(108.3,18){\sx{.7}{e}}
\put(101,23){\sx{.7}{$2$}}
\put(121,23){\sx{.7}{$4$}}
\put(141,23){\sx{.7}{$6$}}
\put(161,23){\sx{.7}{$8$}}
\put(173,23){\sx{.7}{$x$}}
\put(119.5,104.5){\sx{.9}{$y\!=\!b^x$}}
\put(146, 95){\sx{.9}{$y\!=\!\exp_{b,3}^{[c]}(x)$}}
\put(-2,5){\sx{.9}{$y\!=\!\exp_{b,1}^{[c]}(x)$}}
\put(107.6,21.5){\textcolor{green}{\rule{4pt}{1pt}}}
\end{picture}}
\sx{.82}{\begin{picture}(425,70)
\put(25,10){\ing{e1eqzoom4}}
\put( 0,30){\sx{1}{$~0.001$}}
\put( 0,20){\sx{1}{$~ 0$}}
\put( -4,10){\sx{1}{$-0.001$}}
\put(117,1){\sx{1.2}{$\rme\!-\!0.1$}}
\put(226,1){\sx{1.3}{$\rme$}}
\put(318,1){\sx{1.2}{$\rme\!+\!0.1$}}
\put(416,1){\sx{1.2}{$x$}}
\put(325,55){\sx{.96}{$y\!=\! \Re\big(d_{\rm q13}(x)\big)$}}
\put(370,37){\sx{.96}{$y\!=\! \Im\big(d_{\rm q13}(x)\big)$}}
\end{picture}}
}\end{center}
\caption{ Behavior of the two square roots ($c\!=\!\frac{1}{2}$) of  the exponential to base  $b\!=\!\exp(1/\rme)$.
Top picture: Dashed: $y\!=\!\exp_{b,1}^{[1/2]}(x)$ by \eqref{q1};
Thick solid: $y\!=\!\exp_{b,3}^{[1/2]}(x)$ by \eqref{q3};
Thin solid: $y\!=\!\exp_b(x)\!=\!b^{x}$. 
Bottom picture: 
$y\!=\! \Re\big(d_{\rm q13}(x)\big)$, thick line, and 
$y\!=\! \Im\big(d_{\rm q13}(x)\big)$, thin  line,
in vicinity of  $x\!=\!\rme$ by equation \rf{D13}.
\iL{figq}
}
\end{figure}
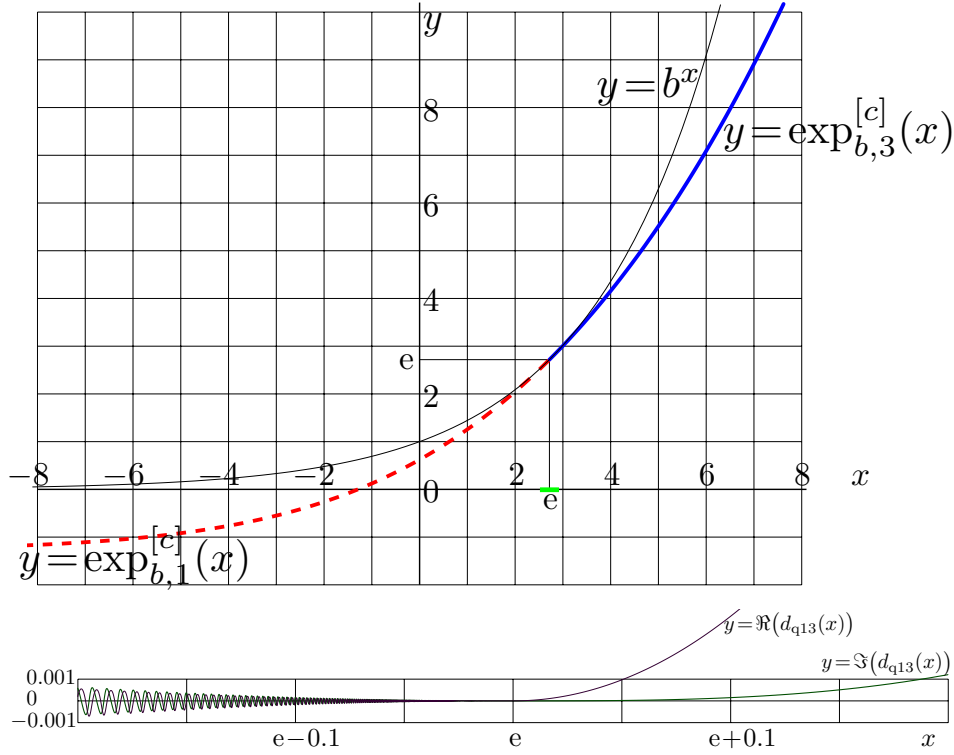
In order to see that $\exp_{b,1}^{[c]}(z)$ and $\exp_{b,3}^{[c]}(z)$, at least for $b\!=\!\exp(1/\rme)$ and $c\!=\!1/2$, are pretty different functions, 
the difference of the "continuations" of these two functions, id est, 
\be
d_{\rm q13}(x)=\exp_{b,1}^{[1/2]}(x\!+\!\rmi o) - \exp_{b,3}^{[1/2]}(x\!+\!\rmi o)
\label{d}
\label{D13}
\eL{dq13}
 in vicinity of $x\!=\!\rme$ is shown at the bottom picture of figure \ref{figsqrt}.

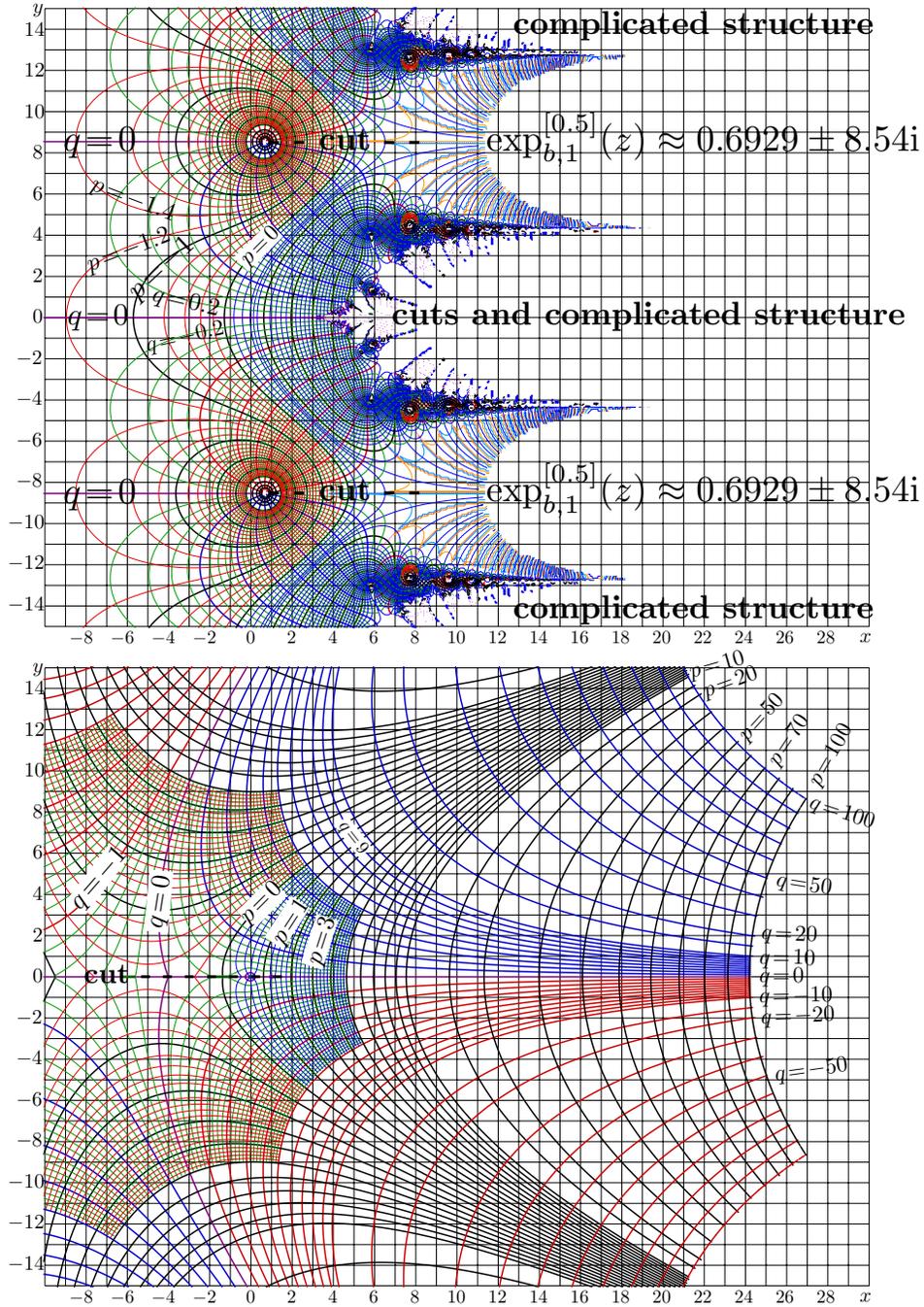
\begin{figure}
\begin{center}
 \sx{.8}{\begin{picture}(415,310) 
\put(0,0){\ing{e1ethalf1} } 
 \figaxe
\put( 10,233){\sx{1.8}{$q\!=\!0$}} 
\put(110,233.4){\sx{1.6}{\bf - - cut - - }} 
\put(210,233){\sx{1.8}{ $\exp_{b,1}^{[0.5]}(z)\approx 0.6929 \pm 8.54 \rmi$} } 
\put( 10,61){\sx{1.8}{$q\!=\!0$}} 
\put(110,62.4){\sx{1.6}{\bf - - cut - -}} 
\put(210,62){\sx{1.8}{ $\exp_{b,1}^{[0.5]}(z)\approx 0.6929 \pm 8.54 \rmi$} } 
\put(229,290){\sx{1.6}{\bf complicated structure}} 
\put(12,148){\sx{1.5}{$q\!=\!0$}} 
\put(170,148){\sx{1.6}{\bf cuts and complicated structure}} 
\put(229,5){\sx{1.6}{\bf complicated structure}} 
\put(23,214){\sx{1.2}{\rot{-21} $p\!=\!-1.4$\ero}}
\put(23,174){\sx{1.2}{\rot{21} $p\!=\!-1.2$\ero}}
\put(46,160){\sx{1.4}{\rot{43} $p\!=\!-1$\ero}}
\put(53,160){\sx{1.2}{\rot{-13} $q\!=\!0.2$\ero}}
\put(51,135){\sx{1.1}{\rot{13} $q\!=\!-0.2$\ero}}
\put(100,174){\sx{1.1}{\rot{45} \textcolor{white}{\rule{22pt}{9pt}}\ero}}
\put(100,177){\sx{1.1}{\rot{45} $p\!=\!0$\ero}}
 \end{picture}}
 \sx{.8}{\begin{picture}(415,320) 
 \put(0,0){\ing{e1eghalf} }  
 \figaxe
\put( 21,182){\sx{1.3}{\rot{53} \textcolor{white}{\rule{26pt}{9pt}}\ero}}
\put( 18,184){\sx{1.3}{\rot{49} $q\!=\!-1$\ero}}
\put( 61,173){\sx{1.3}{\rot{85} \textcolor{white}{\rule{23pt}{9pt}}\ero}}
\put( 58,175){\sx{1.3}{\rot{81} $q\!=\!0$\ero}}
\put(20,148.4){\sx{1.4}{\bf cut - - - - - - -}} 
	\put(316,297){\sx{1.1}{\rot{16}$p\!=\!10$\ero}}
	\put(323,288){\sx{1.1}{\rot{20}$p\!=\!20$\ero}}
	\put(343,266){\sx{1.1}{\rot{48}$p\!=\!50$\ero}}
	\put(358,254){\sx{1.1}{\rot{56}$p\!=\!70$\ero}}
	\put(378,244){\sx{1.1}{\rot{60}$p\!=\!100$\ero}}

\put(373,234){\sx{1.1}{\rot{-18}$q\!=\!100$\ero}}
\put(357,196){\sx{1.1}{\rot{-9}$q\!=\!50$\ero}}
\put(350,170){\sx{1.1}{\rot{-2}$q\!=\!20$\ero}}
\put(349,158){\sx{1.1}{\rot{-1}$q\!=\!10$\ero}}
\put(349,148.8){\sx{1.1}{$q\!=\!0$}}
\put(349,139){\sx{1.1}{\rot{1}$q\!=\!-10$\ero}}
\put(350,129){\sx{1.1}{\rot{2}$q\!=\!-20$\ero}}
\put(357,101){\sx{1.1}{\rot{8}$q\!=\!-50$\ero}}
\put(142,225){\sx{.9}{\rot{-52}  \textcolor{white}{\rule{22pt}{9pt}} \ero }}
\put(144,225){\sx{.9}{\rot{-52} $q\!=\!9$\ero}}
\put( 99,175){\sx{1.2}{ \rot{45}  \textcolor{white}{\rule{19pt}{10pt}} \ero }}
\put(100,177){\sx{1.2}{\rot{45} $p\!=\!0$\ero}}
\put(117,165){\sx{1.2}{\rot{57}  \textcolor{white}{\rule{20pt}{9pt}} \ero }}
\put(117,167){\sx{1.2}{\rot{57} $p\!=\!1$\ero}}
\put(136,155){\sx{1.2}{\rot{72}  \textcolor{white}{\rule{21pt}{9pt}} \ero }}
\put(135,157){\sx{1.2}{\rot{72} $p\!=\!3$\ero}}
 \end{picture}}
\end{center}
\caption{
Two different "square roots" of the exponential to base $\exp(1/\rme)$ by equations (\ref{q1}) and (\ref{q3}) in the same notations as in figures \ref{figmapf} and \ref{figi}.
\iL{figsqrt} $\!\!\!\!\!\!\!\!$
}
\end{figure}

The complex map of the two square roots of the exponential to base $\exp(1/\rme)$ are potted in figure \ref{figsqrt} in the same notations, as in figures \ref{figmapf} and \ref{figi}.
The additional levels $p\!=\!0.692907175521155$ and $q\!=\!\pm 8.53$
are plotted in order to reveal the behavior of the $\exp_{b,1}^{[0.5]}(z)$ in the strips along the cut lines
$\Re(z)\!>\!0.7$, $\Im(z)\!\approx\! \pm 8.5$~. The comparison of the
pictures in figure  \ref{figsqrt} show how different the functions
$\exp_{b,1}^{[0.5]}$  and $\exp_{b,3}^{[0.5]}$ are when considered away from the real axis.
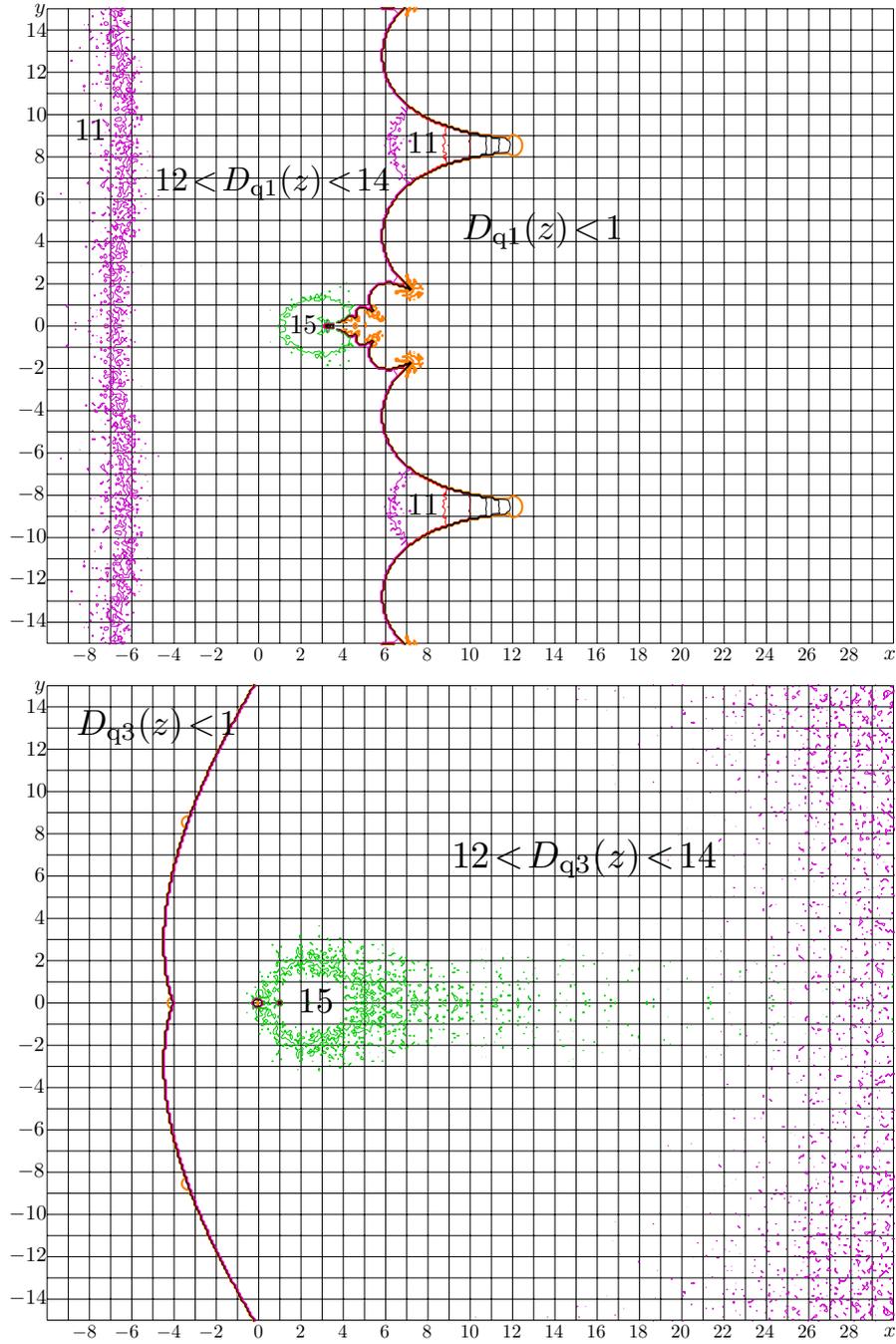
\begin{figure}
\begin{center}
 \sx{.8}{\begin{picture}(415,310) 
 \put(0,0){\ing{e1ethalf2e} } 
 \figaxe 
 \put(14,240){\sx{1.5}{$11$}}
 \put(52,216){\sx{1.6}{$12\!<\!D_{\rm q1}(z)\!<\!14$}}
 \put(172,234){\sx{1.5}{$11$}}
 \put(116,148){\sx{1.3}{$15$}}
 \put(172,62){\sx{1.5}{$11$}}
 \put(198,193){\sx{1.7}{$D_{\rm q1}(z)\!<\!1$}}
 
\end{picture}}
 \sx{.8}{\begin{picture}(415,320) 
 \put(0,0){\ing{e1eghalf2e} }  
 \figaxe
 \put(15,279){\sx{1.7}{$D_{\rm q3}(z)\!<\!1$}}
 \put(120,147){\sx{1.7}{$15$}}
 \put(193,216){\sx{1.8}{$12\!<\!D_{\rm q3}(z)\!<\!14$}}
 \end{picture}}
\end{center}
\caption{
The agreements $D_{\rm q1}(z)$ and $D_{\rm q3}(z)$ by equations (\ref{Dq1}) and (\ref{Dq3})  in the complex plane $z\!=\!x\!+\!\rmi y$. 
\label{sqrtsqrte}
}
\end{figure}
For the functions 
$\exp_{b,1}^{[1/2]}(z)$ and
$\exp_{b,3}^{[1/2]}(z)$, in wide ranges of $z$, the relations
\be
\exp_{b,1}^{[1/2]}\Big(\exp_{b,1}^{[1/2]}(z)\Big)=z
\\
\exp_{b,3}^{[1/2]}\Big(\exp_{b,3}^{[1/2]}(z)\Big)=z
\ee
hold. As in the case of the square root of the logistic operator \cite{logistic}, the ranges of validity of these equations do not cover the whole complex plane, and they are different. In order to show these ranges, the agreements
\be
D_{\rm q1}(z)= \lg \left| \frac
{\exp_{b,1}^{[1/2]}\Big(\exp_{b,1}^{[1/2]}(z)\Big)+\exp_b(z)}
{\exp_{b,1}^{[1/2]}\Big(\exp_{b,1}^{[1/2]}(z)\Big)-\exp_b(z)}
\right|
\iL{Dq1}
\\
\nonumber
\\
D_{\rm q3}(z)= \lg \left| \frac
{\exp_{b,3}^{[1/2]}\Big(\exp_{b,3}^{[1/2]}(z)\Big)+\exp_b(z)}
{\exp_{b,3}^{[1/2]}\Big(\exp_{b,3}^{[1/2]}(z)\Big)-\exp_b(z)}
\right|
\iL{Dq3}
\ee
are shown in figure \ref{sqrtsqrte}. In particular, $\exp_{b,3}^{[1/2]}(x)$ can be considered as "true" square root of the exponential to base $b\!=\!\exp(1/\rme)$
for $x\!<\! \rme$, while $\exp_{b,3}^{[1/2]}(x)$ can be considered at
the "true" root at $x\!>\! \rme$; but these square roots can not be
combined into the same holomorphic function. 

The fractional iterations by \eqref{q1},\eqref{q1} can be evaluated for complex values of $z$ and even 
for complex values of $c$; but only at integer values of $c$ these two functions can be considered as holomorphic extensions of each other.
The fractional iteration provides a smooth (holomorphic) transition from the exponential at $c\!=\!1$ to the 
logarithm at $c\!=\!-1$, passing through the ``square root'' of the
exponential at $x\!=\!1/2$, the identity function at $c\!=\!0$ and the
``square root'' of the logarithm at $c\!=\!-1/2$~. In a similar way, the complex iterations of a function can be considered.

The non-integer iteration of the exponentials provides a set of functions that grow up faster than any polynomial but slower than any exponential. Such functions may find applications in various areas of physics and technology.
Similar non-integer iteration for other functions (including the exponentials of different bases and factorial) were discussed recently 
\cite{citeulike:4195962,kouznetsov:sqrt2,qfac,logistic},
but the peculiarity of the fixed points of the exponential at the base
$b\!=\!\exp(1/\rme)$ required the special consideration above.

\section{Comparison of the tetrationals to different bases}
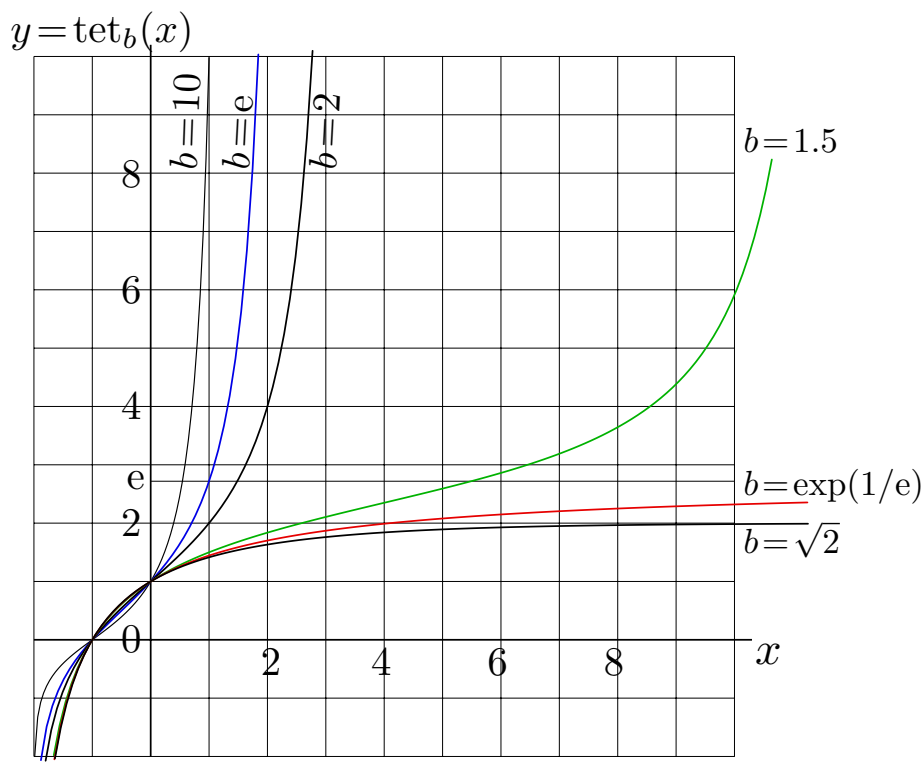
\begin{figure}
\begin{center}
\sx{2.2}{\begin{picture}(155,128)
\put(0,0){\ing{fig10q2}}
\put(-2,125){\sx{.7}{$y\!=\!{\rm tet}_{b}(x)$}}
\put( 17,100){\sx{.7}{$8$}}
\put( 17,80){\sx{.7}{$6$}}
\put( 17,60){\sx{.7}{$4$}}
\put( 18,48){\sx{.7}{e}}
\put( 17,40){\sx{.7}{$2$}}
\put( 17,20){\sx{.7}{$0$}}
\put( 41,16){\sx{.7}{$2$}}
\put( 60,16){\sx{.7}{$4$}}
\put( 80,16){\sx{.7}{$6$}}
\put(100,16){\sx{.7}{$8$}}
\put(126,18){\sx{.8}{$x$}}
\put(30,103){\sx{.68}{\rot{88} $b\!=\!10$ \ero } }
\put(39,103){\sx{.68}{\rot{87} $b\!=\! \rme$ \ero } }
\put(54,103){\sx{.68}{\rot{87} $b\!=\!2$ \ero } }
\put(124,106){\sx{.6}{$b\!=\!1.5$}}
\put(124,47){\sx{.6}{$b\!=\!\exp(1/\rme)$}}
\put(124,37){\sx{.6}{$b\!=\!\sqrt{2}$}}
\end{picture}}
\end{center}
\caption{Tetrationals to base 
$b\!=\!10$,
$b\!=\!\rme\!\approx\! 2.71$,
$b\!=\!2$,
$b\!=\!1.5$,
$b\!=\!\exp(1/\rme)\!\approx\! 1.44$ (the same curve as $F_{1}$ in figure \ref{fig1e1fre}),  
and
$b\!=\!\sqrt{2}\!\approx\! 1.41$
versus real argument.
~\iL{fig10q2}$\!\!\!\!$
}
\end{figure}

In this section, the tetrational $F_1$ to base $\exp(1/\rme)$ is compared to tetrationals to various bases. In fig \ref{fig10q2}, the tetrational ${\rm tet}_{b}$ versus real argument is shown for 
$b\!=\!10$,
$b\!=\!\rme\!\approx\! 2.71$,
$b\!=\!2$,
$b\!=\!1.5$,
$b\!=\!\exp(1/\rme)\!\approx\! 1.44$
and
$b\!=\!\sqrt{2}\!\approx\! 1.41$~.
The functions for $b\!>\!\exp(1/\rm e)$ are evaluated using the Cauchy algorithm described in \cite{citeulike:4195962}.
For $b\!<\!\exp(1/\rme)$, the regular iteration described in 
\cite{kouznetsov:sqrt2}
 is used.
For $b\!=\!\exp(1/\rm e)$, the tetrational is just $F_{1}$ shown also in figure \ref{fig1e1fre}).

\begin{figure}
\noindent\hskip -9pt
\sx{1.3}{\begin{picture}(90,90) \put(0,0){\ing{b15zoom}}
\put(2,85){\sx{.6}{$y$}}
\put(85,.8){\sx{.6}{$x$}} 
\put(74,51.6){\sx{.6}{$q\!=\!0.2$}}
\put(74,43.6){\sx{.66}{$q\!=\!0$}}
\put(73,35.6){\sx{.6}{$q\!=\!-0.2$}}
\put(47,20){\sx{.6}{\rot{46}$p\!=\!1.6$\ero}}
\put(53,15){\sx{.6}{\rot{48}$p\!=\!1.8$\ero}}
\put(66,14){\sx{.6}{\rot{64}$p\!=\!2$\ero}}
\end{picture}}
\sx{1.3}{\begin{picture}(90,90) \put(0,0){\ing{be1ezoom}}
\put(2,85){\sx{.6}{$y$}}
\put(85,.8){\sx{.6}{$x$}} 
\put(68,56){\sx{.6}{\rot{24}$q\!=\!0.2$\ero}}
\put(74,43.6){\sx{.66}{$q\!=\!0$}}
\put(68,32){\sx{.6}{\rot{-28}$q\!=\!-0.2$\ero}}

\put(47,22){\sx{.6}{\rot{38}$p\!=\!1.6$\ero}}
\put(53,17){\sx{.6}{\rot{32}$p\!=\!1.8$\ero}}
\put(58,9){\sx{.6}{\rot{17}$p\!=\!2$\ero}}
\end{picture}}
\sx{1.3}{\begin{picture}(77,90) \put(0,0){\ing{bq2zoom}}
\put(2,85){\sx{.6}{$y$}}
\put(85,.8){\sx{.6}{$x$}}
\put(73,66){\sx{.6}{\rot{64}$q\!=\!0.2$\ero}}
\put(74,43.6){\sx{.66}{$q\!=\!0$}}
\put(72,25){\sx{.6}{\rot{-65}$q\!=\!-0.2$\ero}}

\put(47,22){\sx{.6}{\rot{31}$p\!=\!1.6$\ero}}
\put(52,18){\sx{.6}{\rot{12}$p\!=\!1.8$\ero}}
\put(54,8){\sx{.6}{\rot{-10}$p\!=\!2$\ero}}

 \end{picture}}
\caption{Tetrationals ${\rm tet}_{b}(x\!+\!\rmi y)=p\!+\!\rmi q$ for various $b$, from left to right:
$b\!=\!1.5$,
$b\!=\!\exp(1/\rme)\!\approx\! 1.44$ (the central part of map of $F_{1}$
at the top picture in figure \ref{figmapf}),  
and $b\!=\!\sqrt{2}\!\approx\! 1.41$ in the $x,y$ plane.
~\iL{fig:poltora}$\!\!\!\!$
}
\end{figure}
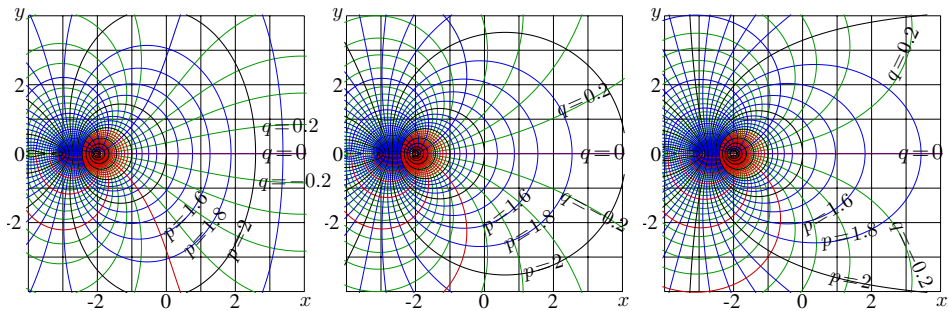

At moderate values of argument or order of unity or smaller, the curves for $b\!=\!1.5$, $b\!=\!\exp(1/\rme)$ and $b\!=\!\sqrt{2}$ are very close. In order to see the difference, the complex maps are shown in figure
\ref{fig:poltora} for $b\!=\!1.5$, left,~ for $b\!=\!\exp(1/\rm e)$,
central,~ and for $b\!=\!\sqrt{2}$, right. The central picture of
figure \ref{fig:poltora} is just a zoom-in from the central part of the top picture in figure \ref{figmapf}. The efficient algorithms of the computation allow to plot all the figures with some reserve of resolution; at the online version, they still can be zoomed-in.

Figure \ref{fig:poltora} indicates no qualitative change of the
tetrational at small variation of the base in vicinity of value
 $b\!=\!  \exp( 1/ \rme )$. 
In particular, within the loop $p\!=\!\Re\big( {\rm tet}_b(z)\big)\!=\!1$ ~ (this loop goes through the origin of coordinates in all the three pictures), the zooming-in of the central parts of pictures in figure \ref{fig:poltora} is necessary to see the difference. 

\newcommand \scalr {
\put(370,783){\sx{6}{$y$}}
\put(370,583){\sx{6}{$4$}}
\put(370,383){\sx{6}{$2$}}
\put(370,183){\sx{6}{$0$}}
\put( 600,150){\sx{6}{$2$}}
\put( 800,150){\sx{6}{$4$}}
\put(1000,150){\sx{6}{$6$}}
\put(1200,150){\sx{6}{$8$}}
\put(1420,150){\sx{6}{$x$}}
}

\begin{figure}
\begin{center}
\sx{.23}{\begin{picture}(1400,810)
\put(0,-10){\ing{fig14a}}\scalr
\put(0,620){\sx{10}{ \color{white}{\rule{24pt}{12pt}} }}
\put(10,640){\sx{10}{ $b\!=\!\rme$}}

\put( 455,710){\rot{85}\sx{4.5}{$c\!=\!2$}\ero}
\put( 565,710){\rot{80}\sx{4.5}{$c\!=\!1$}\ero}
\put( 620,700){\rot{75}\sx{4.4}{$c\!=\!0.9$}\ero}
\put( 700,700){\rot{65}\sx{4.4}{$c\!=\!0.5$}\ero}
\put( 840,710){\rot{50}\sx{4.4}{$c\!=\!0.1$}\ero}
\put( 925,720){\rot{45}\sx{4.6}{$c\!=\!0$}\ero}
\put(1080,710){\rot{40}\sx{4.4}{$c\!=\!-0.1$}\ero}
\put(1300,590){\rot{18}\sx{4.4}{$c\!=\!-0.5$}\ero}
\put(1330,450){\rot{10}\sx{4.5}{$c\!=\!-0.9$}\ero}
\put(1330,390){\rot{ 9}\sx{4.6}{$c\!=\!-1$}\ero}
\put(1330,250){\rot{ 5}\sx{4.6}{$c\!=\!-2$}\ero}

\put(430,445){\rot{73}\sx{4.3}{$y\!=\!\exp(\rme^x)$}\ero}
\put(495,445){\rot{73}\sx{4.4}{$y\!=\!\exp(x)$}\ero}
\put(595,445){\rot{63}\sx{4.4}{$y\!=\!\sqrt{\exp}(x)$}\ero}

\put(840,460){\rot{17}\sx{4.6}{$y\!=\!\sqrt{\ln}(x)$}\ero}
\put(890,330){\rot{ 8}\sx{4.6}{$y\!=\!\ln(x)$}\ero}
\put(890,260){\rot{ 5}\sx{4.6}{$y\!=\!\ln(\ln(x))$}\ero}

\put(14,315){\rot{1}\sx{4.8}{$c\!=\!2$}\ero}

\put(14,210){\rot{4}\sx{4.8}{$c\!=\!1$}\ero}
\put(14,155){\rot{4}\sx{4.7}{$c\!=\!0.9$}\ero}
\put(20, 95){\rot{ 6}\sx{4.7}{$c\!=\!0.5$}\ero}
\put( 90, 10){\rot{23}\sx{4.7}{$c\!=\!0.1$}\ero}

\put(180,-30){\rot{45}\sx{5.5}{$y\!=\!x$}\ero}

\put(280,-30){\rot{53}\sx{4.7}{$c\!=\!-0.1$}\ero}
\put(370,-10){\rot{68}\sx{4.5}{$c\!=\!-0.5$}\ero}
\put(450, 0){\rot{69}\sx{4.7}{$c\!=\!-1$}\ero}
\put(560,0){\rot{70}\sx{4.5}{$c\!=\!-2$}\ero}
\end{picture}}

\sx{.23}{\begin{picture}(1400,880)
\put(0,-10){\ing{fig14b}}\scalr
\put(-30,620){\sx{10}{ \color{white}{\rule{32pt}{12pt}} }}
\put(-20,640){\sx{10}{ $b\!=\!\rme^{1/\rme}$}}
\put( 374,456){\sx{6.5}{$\rme$}}
\put( 674,158){\sx{6.5}{$\rme$}}

\put( 800,680){\rot{75}\sx{4.5}{$c\!=\!2$}\ero}
\put( 870,750){\rot{65}\sx{4.1}{$c\!=\!1$}\ero}
\put( 910,760){\rot{65}\sx{4.01}{$c\!=\!0.9$}\ero}
\put( 940,760){\rot{58}\sx{4.01}{$c\!=\!0.5$}\ero}
\put(1010,805){\rot{50}\sx{4.01}{$0.1$}\ero}
\put(1050,790){\rot{45}\sx{4.01}{$-0.1$}\ero}
\put(1090,780){\rot{25}\sx{4.3}{$c\!=\!-0.5$}\ero}
\put(1150,760){\rot{19}\sx{4.4}{$c\!=\!-0.9$}\ero}
\put(1220,740){\rot{15}\sx{4.5}{$c\!=\!-1$}\ero}
\put(1280,650){\rot{ 7}\sx{4.5}{$c\!=\!-2$}\ero}

\put(940,630){\rot{19}\sx{4.5}{$y\!=\!\rme \ln(x)$}\ero}
\put(860,560){\rot{13}\sx{4.5}{$y\!=\!\rme \, \ln(\rme \ln(x))$}\ero}

\put(14,318){\rot{1}\sx{5}{$c\!=\!2$}\ero}
\put(14,233){\rot{4}\sx{5}{$c\!=\!1$}\ero}
\put( 290,325){\rot{12}\sx{4}{$y\!=\!\exp(\frac{1}{\rme}\rme^{x/\rme})$}\ero}
\put( 190,250){\rot{14}\sx{4}{$y\!=\!\exp(x/\rme)$}\ero}
\put(14,177){\rot{4}\sx{5}{$c\!=\!0.9$}\ero}
\put(20, 90){\rot{15}\sx{5}{$c\!=\!0.5$}\ero}
\put(120,  0){\rot{33}\sx{5}{$c\!=\!0.1$}\ero}
\put(280,-30){\rot{53}\sx{5}{$c\!=\!-0.1$}\ero}
\put(370,-10){\rot{67}\sx{4.5}{$c\!=\!-0.5$}\ero}
\put(440,-10){\rot{75}\sx{4.5}{$c\!=\!-0.9$}\ero}
\put(498, 0){\rot{79}\sx{4.7}{$c\!=\!-1$}\ero}
\put(570,0){\rot{85}\sx{4.9}{$c\!=\!-2$}\ero}
\end{picture}}

\sx{.23}{\begin{picture}(1400,880)
\put(0,-10){\ing{fig14c}}\scalr
\put(-30,620){\sx{10}{ \color{white}{\rule{32pt}{12pt}} }}
\put(-20,640){\sx{10}{ $b\!=\!\sqrt{2}$}}

\put( 855,705){\rot{75}\sx{4.5}{$c\!=\!2$}\ero}
\put( 900,750){\rot{65}\sx{4.1}{$c\!=\!1$}\ero}
\put( 930,760){\rot{65}\sx{4.01}{$c\!=\!0.9$}\ero}
\put( 960,760){\rot{58}\sx{4.01}{$c\!=\!0.5$}\ero}
\put(1025,805){\rot{50}\sx{4.01}{$0.1$}\ero}
\put(1050,800){\rot{45}\sx{4.01}{$-0.1$}\ero}
\put(1060,780){\rot{25}\sx{4.3}{$c\!=\!-0.5$}\ero}
\put(1110,780){\rot{19}\sx{4.4}{$c\!=\!-0.9$}\ero}
\put(1160,760){\rot{15}\sx{4.5}{$c\!=\!-1$}\ero}
\put(1290,705){\rot{ 7}\sx{4.5}{$c\!=\!-2$}\ero}
\put( 900,610){\rot{12}\sx{4.4}{$y\!=\!\ln_{\sqrt{2}}(\ln_{\sqrt{2}}(x))$}\ero}

\put(650,480){\rot{45}\sx{5}{$c\!=\!-2$}\ero}
\put(730,435){\rot{45}\sx{5}{$c\!=\!2$}\ero}

\put(14,318){\rot{1}\sx{5}{$c\!=\!2$}\ero}

\put(390,340){\rot{10}\sx{5}{$y\!=\!\sqrt{2}^{\sqrt{2}^x }$}\ero}

\put(14,233){\rot{4}\sx{5}{$c\!=\!1$}\ero}
\put(210,260){\rot{7}\sx{5}{$y\!=\!\sqrt{2}^{x}$}\ero}

\put(14,177){\rot{4}\sx{5}{$c\!=\!0.9$}\ero}
\put(20, 90){\rot{15}\sx{5}{$c\!=\!0.5$}\ero}
\put(120,  0){\rot{33}\sx{5}{$c\!=\!0.1$}\ero}
\put(280,-30){\rot{53}\sx{5}{$c\!=\!-0.1$}\ero}
\put(370,-10){\rot{67}\sx{4.5}{$c\!=\!-0.5$}\ero}
\put(440,-10){\rot{75}\sx{4.5}{$c\!=\!-0.9$}\ero}
\put(498, 0){\rot{79}\sx{4.7}{$c\!=\!-1$}\ero}
\put(570,0){\rot{85}\sx{4.9}{$c\!=\!-2$}\ero}
\end{picture}}

\end{center}
\begin{caption}{
The iterated exponentials $y\!=\!\exp_b^{[c]}(x)$ versus real $x$ for various real $c$, for bases 
$b\!=\!\rme$ , $b\!=\!\exp(1/\rme)$ and $b\!=\!\sqrt{2}$.
\iL{fig14}
}
\end{caption}
\end{figure}

\section{Comparison of the non-integer iterates of the exponential to different bases}
The non-integer iterates $\exp_b^{[c]}(x)$ are shown in figure
\ref{fig14} versus $x$ for $b\!=\!\rme$, $b\!=\!\exp(1/\rme)$,
$b\!=\!\sqrt{2}$ and various values of $c$. At
$1\!<\!b\!\le\exp(1/\rme)$ there are two $c$-iterates of the
exponential. The one valid for $x$ below the upper fixed point is
shown with a dashed line; the one valid for $x$ above the lower fixed point is drawn with a solid line. The curves for $b\!=\!\exp(1/\rme)$, $c\!=\!1/2$ and $c\!=\!1$ are the same as in top picture of figure \ref{figq}. For the evaluation of the iterated exponential to base  $b\!=\!\rme$, the approximation by \cite{vladi} is used (although the direct method by \cite{citeulike:4195962} could be used instead); the two iterated exponentials to base  $b\!=\!\sqrt{2}$ were evaluated with the algorithms described in \cite{kouznetsov:sqrt2}. In the last case, in the range $2\!<\!x\!<\!4$ both the iterated exponentials are valid, and the deviation between the two iterated exponentials is of order of $10^{-24}$; however, even in this case they are not holomorphic extensions of each other. (Being plotted on a paper, the distance between the two curves for the same $c$ is small not only in comparison to the size of an atom, but also small being compared to the size of atomic nuclei.)  

\section{Conclusion and prospects}
We present some theory of regular iteration and apply it to
the case $f(z)=\rme^{z/\rme}$. 
We extract a quite efficient algorithm to calculate the two
super-logarithms to base $\rme^{1/\rme}$ compared with various
other (standard and separate) methods.
We suggest an efficient new non-polynomial approximation for the two
super-exponentials to base $\rme^{1/\rme}$. 

One of the two super-exponentials, the tetrational $F_{1}={\rm
  tet}_{\exp(1/\rme)}$, is holomorphic in the range $\C\wo\{x\in \R : x\le -2\}$ and 
strictly increasing along the real axis $>-2$. It asymptotically
approaches the limiting value $\rme$. 
The function approaches the same value $\rme$ also in any other
directions, i.e. at large values of $|z|$. The jump at the cut $z\!\le\!-2$ reduces to zero, as $|z|\rightarrow \infty$. 

The other super-exponential $F_{3}$ is an entire function. Along the real axis it 
is strictly increasing
from the limiting value $\rme$ at $-\infty$  to infinity, growing faster than any exponential.
Outside the positive part of the real axis, $F_{3}(z)$ approaches
$\rme$ at $|z|\rightarrow \infty$ in a similar way as $F_{1}$ does. 

Efficient calculation algorithms and portraits for bases
$b\!>\!\rme^{1/\rme}$ and $1\!<\!b\!<\!\rme^{1/\rme}$~
were already provided in
\cite{citeulike:4195962,kouznetsov:sqrt2}; so the whole range $b\!>\!1$
is now covered. 
The plots of the tetrational ${\rm tet}_b$ for $b\!=\!\rme^{1/\rme} \! \approx \! 1.44$ look similar to those for 
 $b\!=\! 1.5$ and those for $b\!=\! \sqrt{2}\!\approx\! 1.41$~;
one may expect that at any fixed value of $z$ from some range, the tetrational ${\rm tet}_b(z)$ is a continuous function of $b$ at least for $b\!>\! 1$. 
This raises the following question: 

Are the holomorphic tetrationals constructed in \cite{kouznetsov:sqrt2}, here and in 
\cite{Trappmann:uniqueness}
(generalization of \cite{Kneser:ReelleAnalytischeLoesungen} which is conjectured to be the super-exponential in \cite{citeulike:4195962}) 
--- which together cover the base range $(1,\infty)$ --- analytic as a function of the base $b$, particularly in the point $b=\rme^{1/\rme}$? If so, what is the range of holomorphism? 
If not: can one obtain an operation 
$(b,z)\mapsto {\rm tet}_b(z)$
defined for $b$ in a vicinity of $\rme^{1/\rme}$ such that for each $b$ the function 
$z\mapsto {\rm tet}_b(z)$
is a real-analytic tetrational on $(-2,\infty)$ and the function 
$b\mapsto {\rm tet}_b(z)$
is holomorphic for each $z$?

There is a similar bifurcation base $b\approx 1.6353$ for the
tetrational ${\rm tet}_b(z)$ as the bifurcation base $e^{1/e}$ is for
the exponential (i.e.\ where the two fixed points change into no
fixed point). One could apply the same methods we used to obtain the
super-exponentials to also obtain a super-tetrationals/pentationals.

\section*{Acknowledgement}
Authors thank the participants of the Tetration Forum\\
\url{http://math.eretrandre.org/tetrationforum/index.php}\\
for stimulating discussions.

\bibliographystyle{amsplain}
\bibliography{main}

\end{document}